\documentclass[10pt]{amsart} 

\usepackage{amsmath,amstext,amscd, amssymb, color, mleftright} 
\usepackage{pinlabel}
\usepackage[normalem]{ulem}
\usepackage[colorlinks=true, pdfstartview=FitV, linkcolor=blue, citecolor=blue, urlcolor=blue]{hyperref}
 \usepackage[abs]{overpic}
\usepackage{xcolor,varwidth}
\usepackage{enumitem, scalerel}

\makeatletter
\def\@tocline#1#2#3#4#5#6#7{\relax
\ifnum #1>\c@tocdepth 
  \else 
    \par \addpenalty\@secpenalty\addvspace{#2}%
\begingroup \hyphenpenalty\@M
    \@ifempty{#4}{%
      \@tempdima\csname r@tocindent\number#1\endcsname\relax
 }{%
   \@tempdima#4\relax
 }%
 \parindent\z@ \leftskip#3\relax \advance\leftskip\@tempdima\relax
 \rightskip\@pnumwidth plus4em \parfillskip-\@pnumwidth
 #5\leavevmode\hskip-\@tempdima #6\nobreak\relax
 \ifnum#1<0\hfill\else\dotfill\fi\hbox to\@pnumwidth{\@tocpagenum{#7}}\par
 \nobreak
 \endgroup
  \fi}
\makeatother
\let\oldtocsection=\tocsection
\let\oldtocsubsection=\tocsubsection
\let\oldtocsubsubsection=\tocsubsubsection
\renewcommand{\tocsection}[2]{\hspace{0em}\oldtocsection{#1}{#2}}
\renewcommand{\tocsubsection}[2]{\hspace{1em}\oldtocsubsection{#1}{#2}}
\renewcommand{\tocsubsubsection}[2]{\hspace{2em}\oldtocsubsubsection{#1}{#2}}

\definecolor{cerulean}{rgb}{0,.48,.65} 
\definecolor{magenta}{rgb}{.5,0,.5} 
\definecolor{dred}{rgb}{.5,0,0} 
\definecolor{green}{rgb}{0,.5,0} 
\definecolor{blue}{rgb}{0,0,1} 
\definecolor{black}{rgb}{0,0,0} 
\definecolor{dgreen}{rgb}{0,.3,0} 
\definecolor{vdred}{rgb}{.3,0,0} 
\definecolor{red}{rgb}{1,0,0} 
\definecolor{salmon}{rgb}{0.98,0.50,0.45} 
\definecolor{gray}{rgb}{.5,.5,.5} 
\definecolor{seagreen}{rgb}{0.13,0.70,0.67} 
\definecolor{chartreuse}{rgb}{0.40,0.80,0.00}
\definecolor{cornflower}{rgb}{0.39,0.58,0.93} 
\definecolor{gold}{rgb}{0.65,0.45,0.00}
 
\theoremstyle{plain}

\newtheorem{thm}{Theorem}[section]
\newtheorem{lemma}[thm]{Lemma}

\newtheorem{cor}[thm]{Corollary}

\newtheorem*{conjproblem*}{The conjugacy problem}
\newtheorem*{0twistedproblem*}{The 0-twisted-conjugacy problem}
\newtheorem*{Htwistedproblem*}{The H-twisted conjugacy problem}
\newtheorem*{Itwistedproblem*}{The I-twisted conjugacy problem}

\theoremstyle{definition}

\newtheorem{remark}[thm]{Remark}
\newtheorem{Remark}[thm]{Remark}

\newtheorem{Open questions}[thm]{Open questions}
\newtheorem{Open question}[thm]{Open question}
\newtheorem{Open problems}[thm]{Open problems}
\newtheorem{Open problem}[thm]{Open problem}

\def\Bbb{\mathbb}
\def\bar{\overline}

\def\Z{\Bbb{Z}}
\def\N{\Bbb{N}}

\def\ni{\noindent}

\def\Dist{\hbox{\rm Dist}}
\def\Area{\hbox{\rm Area}}
\def\dehn{\hbox{\rm Dehn}}
\def\Cay{\hbox{\rm Cay}}

\def\Width{\hbox{\rm Width}}

\def\CAT{\hbox{\rm CAT}}
\def\Ker{\hbox{\rm Ker}}

\def\CL{\hbox{\rm CL}}
\def\PCL{\hbox{\rm PCL}}

\def\F+L{\hbox{$\textup{F}\!_+\textup{L}$}}

\def\ssm{\smallsetminus}

\def\SL{\hbox{\rm SL}}

\def\onto{{\kern3pt\to\kern-8pt\to\kern3pt}}

\def\<{\langle}
\def\>{\rangle}
\def\|{{\ |\ }}

\def\G{\Gamma}

\newcommand{\set}[1]{\left\{#1\right\}}

\newcommand{\abs}[1]{\left|#1\right|}
\renewcommand{\ni}{\noindent}

\def\*{^{\star}}

\def\l{\lambda}

\begin{document}

\title[Conjugator length in finitely presented groups]{Conjugator length in finitely presented groups} 

\author{M.\ R.\ Bridson, T.\ R.\ Riley and A.\ W.\ Sale}

\date \today  

\begin{abstract}
\ni The conjugator length function of a finitely generated group $G$ gives the minimal upper bound on the length of a conjugator for a pair of words that represent conjugate elements in $G$, as a function of the sum of the lengths of the words.   Here, we seek  to promote the systematic  study of conjugator length functions by explaining their significance,  by surveying what is known about them and by explaining fundamental techniques and examples.

  \smallskip
\ni \footnotesize{\textbf{2020 Mathematics Subject Classification:  20F65, 20F10, 20F06}}  \\ 
\ni \footnotesize{\emph{Key words and phrases:} conjugacy problem,  conjugator length}
\end{abstract}

\thanks{The first author thanks the mathematics department of Stanford University for its hospitality and support. The second author gratefully acknowledges the financial support of the Simons Foundation (TRR--Simons Collaboration Grant 318301) and the National Science Foundation (TRR, NSF GCR-2428489).  For the purpose of open access, the authors have applied a CC BY public copyright licence to any author accepted manuscript arising from this submission. ORCID: 0000-0002-0080-9059 (MRB),  0009-0004-3699-0322 (TRR), 0000-0001-5698-9020 (AWS).}

\maketitle

\setcounter{tocdepth}{1}
\tableofcontents 


\section{Introduction} \label{intro}

The purpose of this article
 is to promote the systematic study of {\em conjugator length functions},  setting it alongside
the  systematic study of Dehn functions, which emerged from Gromov's pioneering work on geometric group theory in the late 1980s.
Whereas Dehn functions provide the most natural measure of the complexity of a direct approach to the word problem for
finitely presented groups,   conjugator length functions provide the most natural measure of a direct approach to the conjugacy
problem.  We will present foundational results that reveal the geometric nature of conjugator length functions and  describe some of the main techniques used to bound and calculate  these functions.
We will also survey what is known about the conjugator length functions of various classes of groups, and
we will summarize recent advances concerning the challenge of deciding which functions arise as conjugator length functions.

\subsection{The conjugator length function} \label{sec: the defn}

 Suppose $G$ is a group with a finite generating set $A$.  We write $u \! \sim \!  v$ when words $u$ and $v$ in
 the alphabet $A^{\pm 1}$ represent conjugate  elements  of $G$.
 The conjugator length $\CL(u,v)$ of such $u$ and $v$ is the length (see \ref{sec: Notation and terminology}) of a shortest word $w$ such that $uw=wv$ in $G$.
 The \emph{conjugator length function} $\CL : \N \to \N$ is defined so that $\CL(n)$ is the least integer  $N$ with $\CL(u,v) \leq N$ for all words $u$ and $v$ for which $u \! \sim v$ in $G$ and  whose lengths satisfy $|u|+|v|\le n$.    

Because the idea underlying the conjugator length function is so natural, it is hard to point definitively to its first appearance in the literature, but
the first formal definition might be in \cite{BC} where Brick \& Corson
suggested the name  \emph{annular width function}.

The nature of the arguments used to calculate conjugator length functions makes it natural to take words
as the input data for $\CL(u,v)$,  but  $\CL(u,v)$ depends only on the group elements that $u$ and $v$
represent and not on the choice of representing words.

\subsection{Our aims} This article should serve as a starting point for researchers interested in the systematic
study of quantitative approaches to the conjugacy problem in finitely presented groups.  The corresponding study of
the word problem is well-served by introductory articles and surveys such as \cite{Bridson6, BRS, Gersten, Sapir}, with
which we will compare and contrast.  In Section~\ref{why} we explain the  significance of the conjugator length function from various points of view, including
its close relationship with the study of annuli in Riemannian manifolds.  In Section~\ref{Preliminaries} we set out the foundational results that are central to the study of conjugator length functions as geometric invariants of  finitely presented groups:
we review \emph{annular diagrams}, which are the analogues of  van~Kampen diagrams appropriate to the study of conjugacy and we prove a counterpart to  van Kampen's lemma.  In Section~\ref{survey} we  survey the known estimates of conjugator length for prominent families of groups and we describe fundamental examples.  Finally, in Section~\ref{sec: what functions?}, we summarize recent results that address the question of what functions are conjugator length functions of finitely presentable groups.

\subsection{Novel results} \label{results section}

This article, while predominantly a survey,  includes  some original results.  In Theorems~\ref{Heisenberg group thm}  and \ref{Stallings CL} we prove that the conjugator length functions of the 3-dimensional integral Heisenberg group and of Stallings' group grow quadratically. We extend this to higher dimensional integral Heisenberg groups in Remark~\ref{higher diml Heisenberg}.  We give an elementary proof that the  conjugator length functions of     the Baumslag--Solitar groups   $\textup{BS}(1, m) = \langle a,s \mid s^{-1} a s = a^{m} \rangle$ are linear for all $m \geq 2$, which is a special case of a general result of the third author in \cite{Sale3}.  Theorem~\ref{Free-product with amalgam theorem} concerns finitely generated groups $\Lambda$ with  a cyclic subgroup whose distortion function is $d(n)$: we explain how amalgamating  $\Lambda$ with 3-dimensional integral Heisenberg gives a group whose conjugator length function grows at least as fast as $d(n)$.   Theorem~\ref{central distortion versus CL} sets out another construction which promotes distortion functions to conjugator length functions; this theorem achieves matching upper and lower bounds, but its utility is hampered by a hypothesis that is hard to verify.

\subsection{Terminology} \label{sec: Notation and terminology}

A \emph{word} $w$ is a string $a^{\varepsilon_1}_1 \cdots a^{\varepsilon_n}_n$ where each $a_i$ is in some given alphabet and each $\varepsilon_i \in \{ \pm 1\}$.  We write $|w| =n$.  A \emph{cyclic permutation} of $w$ is a word $a^{\varepsilon_{i+1}}_{i+1} \cdots a^{\varepsilon_n}_n a^{\varepsilon_{1}}_{1} \cdots a^{\varepsilon_i}_i$ for some $i \in \{1, \ldots, n \}$.  We say $w$ is \emph{reduced} if there is no $i$ for which $a_i = a_{i+1}$ and $\varepsilon_i = - \varepsilon_{i+1}$, and it is \emph{cyclically reduced} if all its cyclic permutations are reduced.       For an element $g$ of a  group $G$ with generating set $S$, we write $\abs{g}$ or  $\abs{g}_S$ for the  length of the shortest word on $S^{\pm 1}$ representing   $g$.   Our conventions are   $[x,y] :=x^{-1}y^{-1}xy$ and $x^y := y^{-1} x y$.

\section{Why conjugator length functions?} \label{why}

\subsection{The conjugacy problem.}  \label{the CP and algorithms}     
The \emph{word problem} for a finitely generated group $G$ asks for an algorithm that, on input a word $u$ in the generators, declares whether or not $u$ represents the identity in $G$. 
The \emph{conjugacy problem} asks  for an algorithm that, on input words $u$ and $v$, declares whether or not $u \! \sim \! v$ in $G$.  The word problem is the restriction of the conjugacy problem  to the case where $v$ is the empty word.   The systematic study of both problems originates in the work of Max Dehn \cite{Dehn2}.

The Dehn function has gained prominence as the most natural quantification of a direct approach to the word problem
for a finitely presented group.   It is  bounded above by a recursive function if and only if the word problem in the group is solvable \cite[Theorem~2.1]{Gersten}. Correspondingly, for a  finitely generated group $G$ with a solvable word problem, 
 $\CL(n)$ is bounded from above by a recursive function  if and only if there is an algorithm that solves the conjugacy problem in $G$: for the ``only if" direction,  given $u$ and $v$ of total length $n$,   the hypothesized solution to the word problem allows one to check whether $uw=wv$ in $G$, running through  all words $w$ of length at most $\CL(n)$.  This simple procedure  also solves   the \emph{conjugacy search problem},  which asks for an algorithm that, on input a pair of words $u$ and $v$ such that $u \!  \sim \! v$, returns a word $w$ with $uw=wv$ in $G$. 
In the opposite direction,  if there is an algorithm solving the conjugacy problem in $G$, then that algorithm also solves the word problem and we see that $\CL(n)$ is a computable function by arguing as follows: given $n \in \N$,  we can list all the pairs of words $u$ and $v$ of total length at most $n$ such that $u \sim v$ in $G$, and for each of these pairs we can calculate $\CL(u,v)$ by working through all words $w$  in non-decreasing length-order,  checking whether $uw=wv$.   

There are  finitely presented groups which have decidable word problem but undecidable conjugacy problem    \cite{Bokut, Collins1969, Miller1, OS6}. The conjugator length function of such a group is  non-recursive.

The hypothesis that $G$ has a solvable word problem in the above discussion is necessary. In striking recent work,  Gillis \& Wagner~\cite{GiWa2} prove that there are finitely presented groups whose conjugator length functions are recursive --- indeed, quadratic --- and yet their word problems, and therefore also their conjugacy problems, are undecidable. (We return to this in Section~\ref{sec:S-machines}.)

\subsection{The conjugator length function is a group invariant} \label{s: not coarse} Different finite generating sets for the same group give equivalent $\CL$ functions in the following sense.  For $f,g: \mathbb{N} \to \mathbb{N}$, define $f \preceq g$ if there exists $C>0$ such that $f(n) \leq Cg(  Cn+C ) +Cn +C$ for all $n \geq 0$, and define
  $f  \simeq g$ if $f \preceq g$ and $g \preceq f$.  
 
Unfortunately, conjugator length is not a quasi-isometry invariant.  The starkest illustration of this comes from 
the work of Collins \& Miller  \cite{CollinsMiller}.  They exhibited groups which are finitely presented and have decidable word problem and yet the decidability of the conjugacy problem  does not pass to or from index-2 subgroups.  (A finitely generated group $G$ is    quasi-isometric to its  finite-index subgroups, and such a subgroup is finitely presentable if and only if $G$ is.)

\subsection{The combinatorial geometry of the conjugacy problem.}  Despite not being stable under quasi-isometries,
the  conjugator length function is fundamentally geometric in nature. The geometry associated 
to conjugator length is analogous to the rich geometry that connects the word problem for finitely presented groups 
to isoperimetric problems, specifically,  the study of minimal-area discs filling loops in the universal cover 
of any Riemannian manifold with the given group as its fundamental group.  In the setting of the word problem,  the link from algebra to Riemannian geometry passes through the
 intermediary of combinatorial geometry,  where the filling discs appear in the guise
 of  \emph{van Kampen diagrams}, which portray ways of reducing words that represent the identity to the empty word by applying defining relations.   In the context of the conjugacy problem for finitely presented groups, it is
 the geometry of {\em annuli} rather than discs that is relevant and,  correspondingly,  the study of van Kampen diagrams
 is replaced by the study of {\em annular diagrams},  the systematic study of which began in    Schupp's thesis \cite{SchuppThesis} and was pushed further in \cite{LS1, Sch1}.   Whenever words $u$ and $v$ represent conjugate  elements   other than the identity in a finitely presented group $G$, there is an \emph{annular  diagram}, with oriented 1-cells
 labeled  by generators,  such that $u$ labels the outer boundary cycle and $v$ labels the  inner boundary cycle;
if a  word $w$ labels a path in the 1-skeleton from the basepoint of the outer cycle to the basepoint of the
inner cycle,  then   $uw=wv$.  We shall discuss such diagrams in more detail in Section~\ref{Preliminaries}; they 
provide the most important tool in the study of   conjugator length. 

\subsection{The geometry of annuli and the width of free homotopies}

Let $M$ be a smooth compact Riemannian manifold without boundary. Fix a base point $\ast$ on the unit circle $S^1$. Write $\rho_0 \sim \rho_1$ when a pair of loops $\rho_0, \rho_1 : S^1 \to M$ is freely homotopic, which is to say that there is continuous map $H: S^1 \times [0,1] \to M$ with $H(s,i) = \rho_i(s)$ for all $s \in S^1$ and $i = 0,1$.    For such $\rho_0$ and $\rho_1$, define the  \emph{width} $W(\rho_0, \rho_1)$ to be the infimum of the lengths  of the curves $t \mapsto H(\ast,t)$ among all such $H$.  Then define $\Width_M: [0, \infty) \to [0, \infty)$ so that $$\Width_M(\ell) = \inf \set{ W(\rho_0, \rho_1) \mid \text{ loops } \rho_0 \sim \rho_1 \text{ in } M \text{ of total length } |\rho_0| +  |\rho_1| \leq \ell  }.$$
   
In this setting the width function qualitatively agrees with the conjugator length function of the fundamental group of $M$.  To make this precise, we expand the definition of $f \simeq g$ for functions $f,g: \mathbb{N} \to \mathbb{N}$ to apply likewise to functions $f,g: [0,\infty) \to [0,\infty)$, and then we expand it further to compare  functions $f:[0,\infty) \to [0,\infty)$ with functions $ g:\mathbb{N} \to \mathbb{N}$  by extending the latter to be constant on   intervals $[n, n+1)$.

 We remind the reader that for every integer $d\ge 4$, every finitely presented group is the fundamental group of 
 a closed, smooth Riemannian  manifold of dimension $d$.
   
   \begin{thm}\label{t:width}
  If $G$ is the fundamental group of  a closed, smooth  Riemannian manifold  $M$,  then 	$\Width_M \simeq \CL_G$.  
   \end{thm}

\begin{proof}[Proof outline] 
The proof of this theorem is much easier than that of the corresponding theorem concerning Dehn functions, which is
Gromov's Filling Theorem (a detailed proof of which is contained in  \cite{Bridson6}), but it follows a similar outline.  
In order to compare the geometry of $G$ to the geometry of $M$, one  maps the Cayley graph of $G$  into  the universal
cover $\widetilde{M}$, first by fixing a basepoint $x_0$ so that $G$ can be mapped to $\widetilde{M}$ by $g\mapsto g.x_0$;
a mapping of the Cayley graph $\iota: \Cay(G)\to \widetilde{M}$   is then determined by a choice,  for each generator $a\in G$,  of a rectifiable path $\sigma_a$ from $x_0$ to $a.x_0$:  
the edge $[1,a]\subset \Cay(G)$ is sent to $\sigma_a$ and $\iota$ is the $G$-equivariant of this assignment. 
The \v{S}varc-Milnor Lemma \cite[page~140]{BH} assures us that $\iota$  is a Lipschitz map that has a 
quasi-inverse; let the Lipschitz constant be $\lambda$. 

 For words $u,v,w$ in the generators, 
if $uw=wv$ in $G$,  then there is a loop in  $\Cay(G)$ that begins at the vertex $1$ and is labelled $uwv^{-1}w^{-1}$.
The image of this loop under $\iota$ is a loop $\gamma$ in $\widetilde{M}$; 
we retain the labels $u,w, v^{-1},w^{-1}$ on the corresponding arcs of this loop.  As $\widetilde{M}$ is simply-connected, 
$\gamma$  can be filled with a disc. This disc projects to an annulus in $M$ connecting a loop representing $u\in\pi_1(M,x_1)$ to a loop representing $v\in\pi_1(M,x_1)$, where $x_1\in M$ is the image of $x_0\in\widetilde{M}$ under the 
covering map  $\pi:\widetilde{M}
\to M$.  (In this annulus, 
the images of the two  arcs of $\gamma$ labeled $w$ coincide.)
The sum of the lengths of the loops representing $u, v\in \pi_1(M,x_1)$ is at most $\lambda (|u|+|v|)$ and the length of the arc along the annulus connecting them (i.e.\ the image of the path labeled $w$) is at most $\lambda |w|$.  This construction shows that for ``word-like" loops
$U,V:S^1 \to M$ corresponding to words $u,v$ in the generators of $G=\pi_1(M, x_1)$, if $U$ and $V$ are freely homotopic
(equivalently $u\sim v$ in $G$), then   $W(U,  V) \le \lambda\ \CL(u,v)$.  (A loop is  {\em word-like} if it is the image of an
edge-path from $\Cay(G)$ under the composition $\pi\circ\iota$.)

A standard approximation argument in the manner of 
the \v{S}varc-Milnor Lemma, shows there are constants $A, B,C$ such that every loop of length $\ell$ in $M$ is freely homotopic
to a word-like loop of length at most $A\ell + B$ via a homotopy of width at most $C$.  And the lengths of word-like loops in
$M$ differ from the lengths of the words to which they correspond by at most a factor of $\lambda$.  Together, these
observations imply that $\Width_M(n) \preceq \CL_G(n)$.  

For the reverse inequality,   given words $u$ and $v$  with $|u|+|v|\le n$ and
$u\sim v$ in $G$, we consider the word-like loops $U$ and $V$ 
based at $x_1\in M$ that these words label; the sum of the lengths of these loops is at most $\lambda n$. We
connect $U$ to $V$ by a homotopy $H$ of width $L<W(U,  V)+1$.  By definition,
$L$ is the length of the arc $\omega: t \mapsto H(\ast,t)$ that the basepoint traces along the (singular) annulus in $M$ 
that is the image of $H$.  We lift the loop $U\omega V^{-1}\omega^{-1}$ to a loop in $\widetilde{M}$ based at $x_0$. 
We then
replace the first lift of $\omega$ in this loop
by a path with the same endpoints that is the image  under $\iota: \Cay(G) \to \widetilde{M}$ of an edge path -- let
$w$ be the label on this path,  where $|w|$ is chosen to be minimal.  Finally, we replace the other lift of $\omega$ in our 
loop (i.e.~the lift beginning at $x_0$) with the image under $\iota$ of the path labeled $w$ that begins at $1\in \Cay(G)$.
By construction,  $|w| \le \lambda |\omega| \le \lambda\, (W(U,  V)+1) \le \lambda (\Width_M(\lambda n) + 1)$ and
$uw=wv$ in $G$. Therefore $\CL_G(n)\preceq \Width_M(n)$.
\end{proof}

\begin{remark}
We stated Theorem \ref{t:width} as a result about Riemannian manifolds because that captures the correspondence that we want to fix in the reader's mind between conjugator length and the geometry of annuli.  But the proof that we have
sketched uses very little Riemannian geometry: it works equally well for a much larger class of  geodesic spaces, e.g.~piecewise Euclidean complexes with finitely many cells. 
\end{remark}

\subsection{Relationships to other invariants.} \label{sec: rel to other invariants} The conjugator length function is related to the \emph{annular Dehn function}
of a finitely presented group $G$, which was defined by Brick \& Corson in \cite{BC} as the least upper bound on the number of faces (the \emph{area}) required in minimal-area annular diagrams portraying conjugacies $u \sim v$ with $|u|+|v|\le n$.
Brick \& Corson established basic inequalities relating conjugator length,  the annular Dehn function of $G$, and the Dehn function: in short, the annular Dehn function bounds the other two functions from above, but can itself be bounded from above by their composition.   Gillis \& Riley \cite{GiRi} recently exhibited examples showing that these bounds are not tight in general and that the three invariants are independent in the sense that no two determine the third.

Conjugator length complements \emph{conjugacy growth} --- the number of conjugacy classes intersecting the ball of radius $n$ --- whose systematic study was promoted by Guba \& Sapir in \cite{GS1}.    

The \emph{permutation conjugator length function} $\textup{PCL}$ of \cite{AS} is a variant on conjugator length that records the length of a minimal length word  required to conjugate some cyclic permutation of $u$ to some cyclic permutation of $v$.  It is pertinent in instances where it admits sublinear upper bounds,  which we will see in Section~\ref{sec:hyp groups} is the case for hyperbolic groups.     

It seems unreasonable to expect a relationship between the \emph{Dehn function} and the conjugator length function
of an arbitrary finitely presented group,  but it is a challenge to determine the full scope of pairs $(f,g)$ of functions $\N \to \N$ for which there is a finitely presented group with Dehn function $\simeq f$ and conjugator length function $\simeq g$.   This challenge has recently been taken up Gillis \& Wagner in \cite{GiWa2}.  They show that for all pairs of \emph{reasonable} (our terminology) functions $f,g :  \N \to \N$  such that $f(n) \succeq n^2$ and $g(n) \succeq n^3$ there is such a finitely presented group --- here, \emph{reasonable}  means that  $f$ must already be equivalent to the Dehn function of \emph{some} finitely presented group, and $g$ must be positive superadditive (that is, $g(q) \geq 1$ for all $q \in \N \ssm \{0\}$ and $g(p+q) \geq g(p) + g(q)$ for all $p, q \in \N$) and be the time function of an $S$-machine.

\subsection{Equations over groups.} The conjugacy search problem is an example of an equation over a group:  viewing $u$ and $v$ as coefficients, we try to solve for $x$ in the equation $ux = xv$.  This is a fundamental example of a \emph{quadratic} equation in a group.  The solutions to quadratic equations in finitely presented groups  can be represented by diagrams on surfaces \cite{SchuppQuadratic} --- in the case of $ux = xv$,  these are annular diagrams.  The study of equations over finitely 
presented groups is a rich and challenging subject, with important early work on equations over free groups by Makanin \cite{Makanin} and Razborov \cite{Razborov, Razborov2}, which has been extended greatly by Sela~\cite{SelaI} {\em et seq.},
 Kharlampovich \& 
 Myasnikov \cite{KaMy} and others.  Closer to the spirit of this article, in \cite{SelaIII} Sela provides a quantitative analysis of the Makanin--Razborov diagrams that encode solution sets of systems of equations over limit groups.   Kharlampovich \&  Vdovina \cite{KV} give linear estimates for the lengths of solutions of quadratic equations over free groups in terms of the sums of the lengths of the coefficients.  And in  \cite{KMTV}, working with  Mohajeri and Taam, they generalize 
 their results to torsion-free hyperbolic groups and to toral relatively hyperbolic groups.

\subsection{Diophantine problems.}  Diophantine equations are extensively studied in number theory.  Hilbert's tenth problem, which was famously solved in the negative in  the 1970s, asks  whether there is (what we now interpret as) an algorithm which, on input a  polynomial equation with integer coefficients, declares whether or not the equation admits an integer solution.  There has also been a lot of work on  quantifying the complexity of instances of this problem (e.g., Smale's Fifth Problem) and bounding the sizes of solutions, when they exist, in terms of the sizes of the coefficients --- e.g., \cite{BFRT, Kornhauser, LagariasTalk}. This last challenge is of particular relevance in the setting of conjugator length functions.  In our proofs of Theorems~\ref{Heisenberg group thm}  and \ref{Stallings section}  estimates on conjugator length functions come from estimates on the sizes of solutions to a corresponding system of linear Diophantine equations.  More elaborate examples of the interplay between such systems of 
equations and conjugator length functions can be found in \cite{BrRi2, BrRi1}.  Such examples hold out the enticing prospect
that the study of conjugator length functions might provide a   locus of connection between group theory and this area of number theory.

\subsection{Cryptography.}  \label{Cryptography}
The conjugacy search problem  underlies a public key protocol due to Ko, Lee et al.\ \cite{KoLee}.  The idea behind the protocol is that a group $G$ is published together with some element $g \in G$.  The prototype for $G$ is the braid group on $2n$ strings --- group elements are communicated as words in Garside  normal-form   on a standard finite generating set.  Alice and Bob seek to establish a shared private key with which to encrypt their communication.  They privately choose from $G$   one element each,  $a$ and $b$ say,   in such a way that  $a$ and $b$ are sure to commute. For example, in the braid group setting, it is agreed that Alice will take  $a$ to be a braid on the first $n$ strings and Bob will take $b$  to be a braid on the remaining $n$ strings. Alice computes the normal form word for $g^a$ and sends it to Bob, who likewise computes and sends $g^b$  to Alice.  They then both calculate $g^{ab}= g^{ba}$, which serves as their common key.              
 
An eavesdropper tapping into their communication could break this system by recovering $a$ from the pair $g$ and $g^a$ or $b$ from $g$ and $g^b$ and then calculating  $g^{ab}$ for themselves and (modulo an issue with centralizers) this is the conjugacy search problem. 

A   brute-force approach to cracking the conjugacy search problem is to try all words of length $1$, then all of length $2$, and so on.  If conjugators are sure to be sufficiently long, then this  length-based attack is hopeless.   So it would seem that a fast growing conjugator length function for $G$ is desirable.

In truth, the situation is more subtle.  The conjugator length functions  of mapping class groups,    and in particular of braid groups, are linear \cite{MM00, Tao}.  Nevertheless, the number of possible  words that one has to try to find a conjugator grows exponentially with  length,  so this brute-force approach can be prohibitively expensive even without particularly long conjugators.    On the other hand, many cases can by bypassed.  See   \cite{GG-M} and references therein for insights on how this plays out in braid groups.
 
Conjugacy search problems and, therefore, conjugator length functions are significant for other cryptographic protocols.  Anshel-Anshel-Goldfeld \cite{AAG} provides  another influential example.  See the survey \cite{MSU}. 

\subsection{Applications to operator algebras and K-theory} The existence of a  polynomial upper bound on conjugator length is one of a set  of conditions that Ji, Ogle, \& Ramsey \cite{JOR}  require of the groups for which they prove the  $\ell^1(G)$ Idempotent Conjecture  and the $\ell^1$-Stronger-Bass Conjecture.    
The same condition appeared in the K-theory calculations of   \cite{ChakYam}.

\section{Van Kampen's Lemma adapted for conjugacy} \label{Preliminaries}  \label{general simplifications}

\subsection{Van Kampen diagrams and annular  diagrams}

Van Kampen diagrams and annular  diagrams are the basic objects of study in the approach to decision problems
in group theory that proceeds via combinatorial geometry.

For background and further reading, see  \cite{Bridson6, LS1,  SchuppThesis, Sch1} and  \cite[page~454]{BrH}.

Suppose $G$ is a group with finite presentation $\langle A \mid R \rangle$.  A \emph{van~Kampen diagram} over $\langle A \mid R \rangle$
 for a word $w$ on $A^{\pm 1}$ is a   finite,  planar,  contractible, combinatorial 2-complex whose edges are directed and labelled by elements of $A$ in  such a way that around each face (starting from some vertex)  one reads a word in $R^{\pm 1}$ and around the boundary
cycle,  starting from a specified basepoint,  one reads $w$.  Here \emph{planar} means that a planar embedding is fixed, not just that it exists.   Such a diagram portrays a scheme for proving that $w=1$ in $G$ by
expressing it, in the free group $F(A)$,  as a product of conjugates of the defining relations.

\begin{lemma}[van~Kampen's Lemma] \label{vKs lemma} Let $G=\langle A \mid R \rangle$ be a 
 finitely presented group. A word $w$ in the letters $A^{\pm 1}$ represents $1$ in $G$ if and only if  $w$ admits a van~Kampen diagram over $\langle A \mid R \rangle$.
\end{lemma}

Let $\Area(w)$ be the least integer  $N$ such that there is a van~Kampen diagram for $w$ with $N$ faces.  The \emph{Dehn function} $\dehn : \N \to \N$ of $\langle A \mid R \rangle$ is $$\dehn(n) \ := \ \max\set{ \ \Area(w) \mid w =1 \textup{ in } \Gamma  \textup{ and } |w| \leq n \  }.$$  
Equivalently, $\dehn(n)$ is the minimal $N$ such that if  $w =1$ in $G$ and $|w| \leq n$, then $w$ 
is equal in the free group $F(A)$ to a product of  $N$ or fewer conjugates of the defining relations and their inverses.

An \emph{annular diagram} for a pair of words $u$ and $v$ is a finite planar combinatorial 2-complex $\Omega$ homotopy equivalent to a circle or, in the degenerate case  (where $u$ or $v$ is the empty word), a point; its edges are directed and are labelled by  elements of $A$ in such a way that around each face one reads some word in $R^{\pm 1}$, and anticlockwise around the two boundary cycles from vertices  $\star_u$ and $\star_v$ one reads words $u$ and $v$, respectively.   

A van~Kampen diagram or annular diagram is \emph{reduced} when it does not contain an adjacent pair of `cancelling' 2-cells --- that is, a pair of $2$-cells with a common edge such that  around the attaching 2-cycles of each cell, starting with that common edge, one reads the same word, 	clockwise for one and anticlockwise for the other. 	  A  van~Kampen diagram or annular diagram that is of minimal area among all diagrams for its boundary word(s) will be reduced.

Here  is the conjugacy analogue of van~Kampen's Lemma:

\begin{lemma}
	[Cf.\ Lemma~5.2 of \cite{LS1}, also  \cite{SchuppThesis, Sch1}] 
	\label{annular diagrams lemma}
	Suppose $G$ is group with finite presentation $\langle A \mid R \rangle$.
If a pair of words $u$ and $v$ admits an annular diagram, then $u \sim v$ in $G$.  Conversely, if $u \sim v$ in $G$ and
these words  do not represent $1$ in $G$, then  they admit  a non-degenerate annular diagram.
\end{lemma}

\begin{proof} 
Suppose $\Omega$, as above, is an annular diagram  for the pair $u$ and $v$.    
Let  $\rho$ be a simple path in the 1-skeleton of $\Omega$ from $\star_u$ to $\star_v$.  Let $w$ be the word one reads along $\rho$.  Following $\rho^{-1}$ from $\star_v$ to $\star_u$, then following the boundary cycle labeled $u$,  then returning  from $\star_u$ to $\star_v$ along $\rho$, and then  following $v$ around the other boundary cycle gives a loop labeled $w^{-1}uwv^{-1}$.  This loop is contractible in $\Omega$, so $w^{-1}uw=v$ and $u \sim v$ in $\G$. Indeed,  \emph{cutting}  $\Omega$ along  $\rho$  yields a   van~Kampen diagram for $w^{-1}uwv^{-1}$.

For the converse, suppose  $u \sim v$ in $G$.  Then $w^{-1}uwv^{-1} =1$ for some word $w$, and so by van~Kampen's Lemma there is a van~Kampen diagram $\Delta$ for $w^{-1}uwv^{-1}$ over $\langle A \mid R \rangle$.  One might hope to obtain an annular diagram  by identifying the boundary arcs of $\Delta$ labelled $w$ and $w^{-1}$. However, such an identification is  problematic when, for instance, the boundary arcs are not disjoint simple paths.  The following treatment is more careful.

We expand on the proof of     \cite[Lemma~5.2]{LS1}.  
Because $u$ is conjugate to $v$ in $G$, it is equal in $F(A)$ to a word    
\begin{equation} \label{lollipop word}
	 U  \ = \ (w v w^{-1}) \ (w_1  r_1  w_1^{-1}) \ \cdots \  (w_m  r_m   w_m^{-1}) 
\end{equation}
for some $m \geq 0$, some words  $w, w_1, \ldots, w_m$, and some $r_1, \ldots, r_m \in R^{\pm 1}$.

\begin{figure}[ht]
\begin{overpic}
{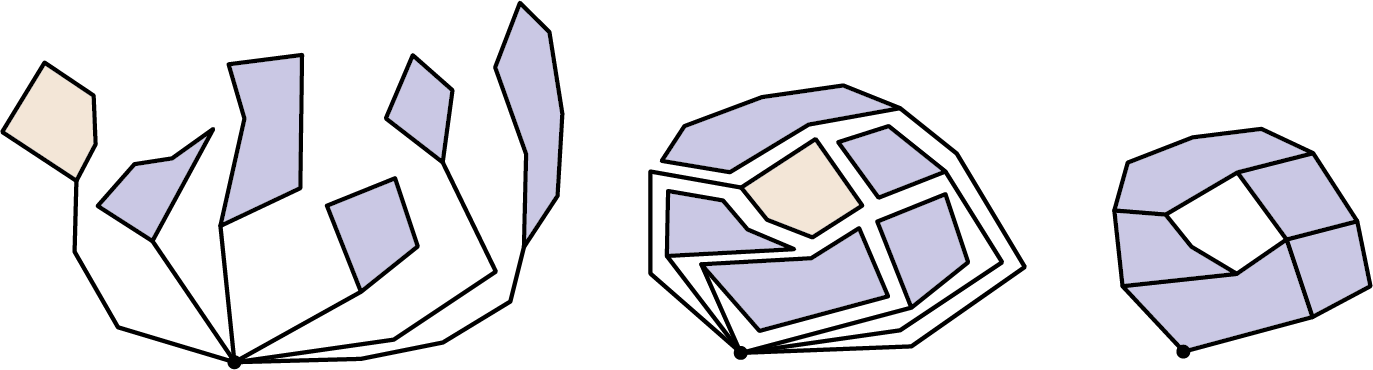}
 \put(0,66){\small{$v$}}     
 \put(35,56){\small{$r_1$}}     
 \put(60,79){\small{$r_2$}}     
 \put(84,48){\small{$r_3$}}     
\put(102,77){\small{$r_4$}} 
\put(131,86){\small{$r_5$}} 
 \put(13,16){\small{$w$}}     
 \put(34,16){\small{$w_1$}}     
 \put(56,25){\small{$w_2$}}     
\put(67,16){\small{$w_3$}}
\put(106,25){\small{$w_4$}}
\put(115,6){\small{$w_5$}}
\end{overpic}
 \caption{Folding up a lollipop diagram to make an annular diagram}
  \label{fig:lollipop}
\end{figure}
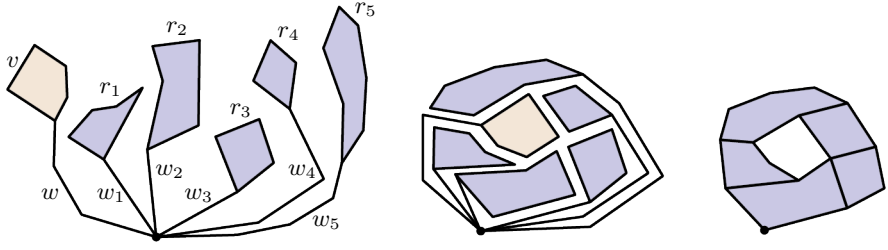

A `lollipop diagram'  for $U$ is a van~Kampen diagram for $U$ over $\langle A \mid R \cup \set{v} \rangle$ with base vertex $\star_u$, and with $m+1$ faces labelled by  $v$, $r_1$, \ldots, $r_m$, each connected to $\star_u$ by an edge-path labelled by $w$,    $w_1$, \ldots, $w_m$, respectively. We seek to form a van~Kampen diagram $\Delta$ for $u$ over $\langle A \mid R \cup \set{v} \rangle$ by folding up this lollipop diagram in the manner of a proof of van~Kampen's lemma  as illustrated in Figure~\ref{fig:lollipop} and explained in detail in \cite{Bridson6}: find a pair of edges $e_1$ and $e_2$ in the boundary circuit that are successive (meaning that the terminal vertex of $e_1$ is the initial vertex of $e_2$) and are labeled by  $a$ and $a^{-1}$, respectively, for some $a \in A^{\pm 1}$; then identify $e_1$ and $e_2$ so as to remove $a a^{-1}$ from the boundary word, as shown left in Figure~\ref{fig:fold}.  We can repeat until the boundary word becomes the reduced form of  $u$. 
If we preserve planarity during this process,  then we can complete the proof by attaching appropriate additional edges to the outer boundary cycle so as to make the boundary word exactly $u$: for example, if $u\equiv u_1aa^{-1}u_2$ and the boundary cycle is $u_1u_2$, then we get a diagram with boundary label $u$ by introducing a new vertex $x$ outside the diagram 
and a new edge labelled $a$  joining the terminus of the arc labeled $u_1$ to $x$.

\begin{figure}[ht]
\begin{overpic}
{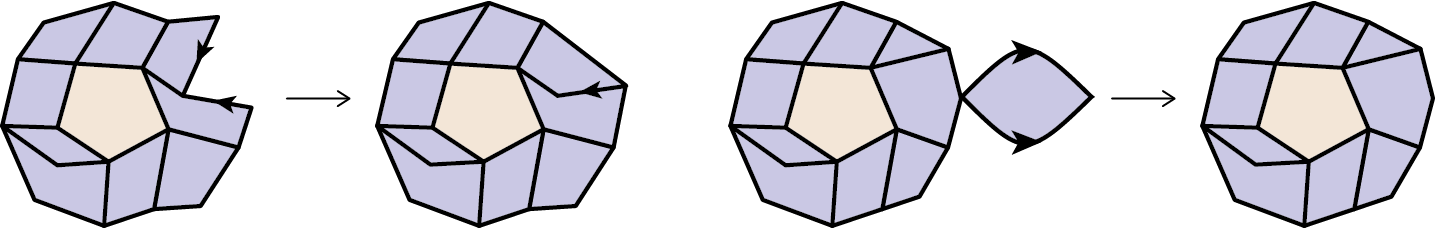}
 \put(48,24){\small{$e_2$}}     
 \put(39.5,42){\small{$e_1$}}     
 \put(56,32.5){\small{$a$}}     
 \put(51,39){\small{$a$}}    
 \put(140,27){\small{$a$}}    
 \put(247,13){\small{$a$}}     
 \put(247,46){\small{$a$}}  
 \put(244,24.5){\small{$e_2$}}     
 \put(244,34.5){\small{$e_1$}}     
\end{overpic}
 \caption{Folding together a pair of edges}
  \label{fig:fold}
\end{figure}
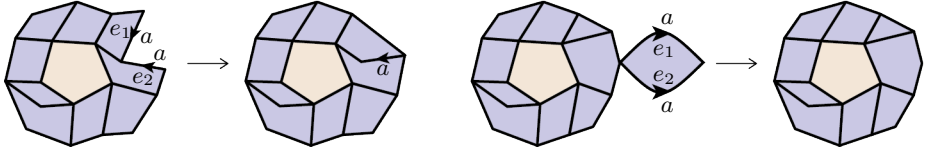

But the proviso about planarity is a serious one:
 if the initial vertex of the edge  $e_1$ and the terminal vertex of $e_2$ are the same, then identifying $e_1$ and $e_2$ will break the diagram's planarity. The remedy is that in this circumstance, we instead remove the subdiagram that $e_1$ and $e_2$ together enclose, as shown on the  right in Figure~\ref{fig:fold}.   Because this operation may be necessary,  some faces of the lollipop diagram may not  give rise to a face in  $\Delta$.   Nevertheless, the faces of $\Delta$ will be a subset of those in the lollipop diagram  (see Remark 4.2.5 in \cite{Bridson6}).  As   $u \neq 1$ in $G$, these faces must include one (and no more than one) with boundary word $v$. (In particular, the vertex $\star_v$ cannot have been removed.)     Removing the interior of this special face gives an annular diagram $\Omega$ for the pair $u$  and $v$ --- we read $u$ and $v$ around its outer and inner boundary cycles from $\star_u$ and  $\star_v$, respectively.  
\end{proof}

The case where $u$ and $v$ represent the identity, excluded in Lemma~\ref{annular diagrams lemma}, is not significant to the study of conjugator length functions --- after all, in that circumstance,  $u=v$ in $G$,  and so  $\CL(u,v)=0$.

\begin{remark} Every van~Kampen diagram admits a  continuous, label-preserving, combinatorial map to the Cayley 2-complex of the presentation. This is not the case for annular diagrams,  unless $u$ and $v$    represent the identity:
there is a continuous, label-preserving, combinatorial map to the presentation 2-complex, but this does not lift to the universal cover.
\end{remark}

We can reinterpret the definition of conjugator length (from Section~\ref{sec: the defn}) in terms of annular diagrams.   

\begin{lemma} \label{lem:annular diagram path}
For words $u$ and $v$ that do not represent the identity in a finitely presented group $G$, 
if $u\sim v$ then $\CL(u,v)$ is the least integer $L$ for which  there is an annular diagram $\Omega$ for the pair $u$ and $v$  in which there  is a path  of length  $L$ from $\star_u$  to $\star_v$ in the 1-skeleton of $\Omega$. 
\end{lemma}

 \begin{proof}
We review our proof of Lemma~\ref{annular diagrams lemma}.  The inequality $\CL(u,v)\le L$ comes from the observation that
if $w$ is the label on the path hypothesized  in $\Omega$, then $uw=wv$ in $G$, as per the proof of Lemma \ref{annular diagrams lemma}.  For the reverse inequality, consider the edge-path    in the lollipop diagram along which we read  the subword $w$ of $U$ from \eqref{lollipop word}. This path, which we also call $\rho$, starts at $\star_u$ and ends at a vertex $\star_v$ on the boundary of the face corresponding to $v$.  We may take $\rho$ to have length equal to the length of a shortest word $w$ such that $uw=wv$ in $G$.  It would be enough to argue that $\rho$ persists through the folding process so as to yield a path in  $\Delta$ of length at most $|w|$ from  $\star_u$  to $\star_v$ in $\Omega$.  The instances of the folding process that  jeopardize this are when a  subdiagram $S$ between $e_1$ and $e_2$ is removed. In that event we remove any subpath of $\rho$ through the interior of $S$.
  	 \end{proof}

 \subsection{$t$-Corridors and $t$-Annuli}\label{subsea:corridors} 

We continue our account of the basic tools used in the diagrammatic approach to decision problems. 

Let  $G=\<A\mid R\>$
be a finitely presented group. Figure~\ref{fig:corridor} illustrates the following definitions.

Suppose some $t \in A$ has the property that every defining relation in which $t^{\pm 1}$ occurs contains exactly one $t$ and one $t^{-1}$.  For example, this is the case in the standard presentation of an  HNN-extension with stable letter $t$.  
Following \cite{BridsonGersten}, we define a \emph{$t$-corridor} in a van~Kampen or annular diagram $\Delta$ to be either
\begin{itemize}
	\item a union of distinct, closed 2-cells 
	$\sigma_1, \ldots, \sigma_m$ such that  $\sigma_i$ shares a $t$-edge with $\sigma_{i+1}$ for $i=1, \ldots, m-1$ and  the other $t$-edges in the boundaries of $\sigma_1$ and $\sigma_m$ lie in $\partial \Delta$,   or
	\item a single $t$-edge in the 1-dimensional part of the diagram; that is,  a $t$-edge  not in the boundary of any 2-cell.
\end{itemize} 
 The length of a $t$-corridor is the  number of $2$-cells it contains.
 A \emph{$t$-annulus} (or \emph{annular $t$-corridor}) is defined similarly,  except that there is an additional $t$-edge  shared by $\sigma_m$ and $\sigma_1$ and no $t$-edge in the closure of the annulus lies in $\partial\Delta$.     A $t$-corridor in an annular diagram is \emph{radial} when it connects a $t$-edge on the inner boundary component to a $t$-edge on the outer boundary component.  Each  $t$-annulus  has a {\em core loop}, which is an embedded loop in $\Omega$ that intersects each open 2-cell of the $t$-annulus in an
 arc that connects the midpoints of the two $t$-edges in the boundary of the cell. 
 A  $t$-annulus  is \emph{essential} when it  encloses the inner boundary; more precisely,  the core loop of the annulus
 is homotopic to each of the boundary cycles of $\Omega$.

\begin{figure}[ht]
\begin{overpic}
{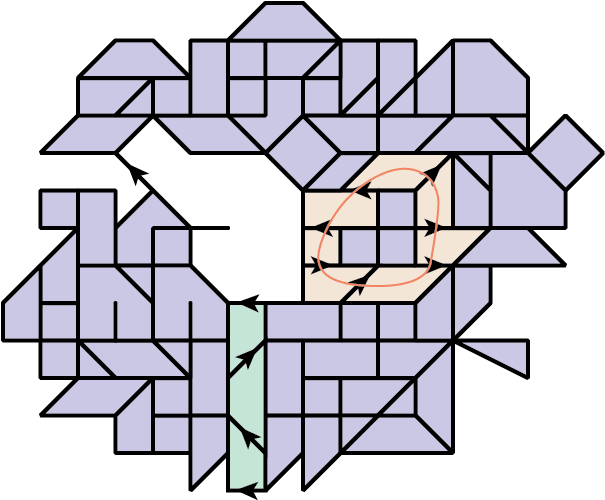}
\end{overpic}
 \caption{Take all the edges in this figure that are decorated with arrows to be $t$-edges.  It then depicts an annular diagram with two radial $t$-corridors, one of length $0$ and one of length $3$, and one annular $t$-corridor, whose core loop is inessential. }
  \label{fig:corridor}
\end{figure}

No  $t$-corridor or $t$-annulus  can cross itself or another $t$-corridor or $t$-annulus.  Thus $t$-corridors and $t$-annuli  partition   diagrams into a
collection of subdiagrams. Many arguments proceed via an analysis of the resulting decomposition;
for example, the famous Lemmas of Britton and Collins (see, for example, \cite{Lyndon}), concerning the word and the conjugacy problems (respectively) in  HNN-extensions,  can be expressed in terms of the possible
configurations of corridors in van~Kampen and annular diagrams.

 \begin{remark} The union of the open 2-cells in a $t$-corridor and the $t$-edges between them form an open disc embedded
 in the ambient diagram $\Delta$, but  it  is important to note that the sides of the corridor need not be embedded arcs, although they are, crucially,  non-crossing paths. Likewise, the two boundary cycles of a $t$-annulus are non-crossing loops.
 \end{remark}

\subsection{HNN-extensions and sub-diagram simplification} \label{sec:removing t-rings}

We work here towards Corollary~\ref{cor:HNN simplify} in which we describe circumstances in which pairs of conjugate words always admit annular diagrams with  no inessential annuli. This is desirable   because it lets us argue that the subdiagrams complementary to the corridors are diagrams over presentations of   groups which are  simpler in that their defining generators and defining relators are strict subsets of those of the original group.   

To get there we will examine (i) \emph{cutting  a subdiagram  out of a diagram} and (ii) \emph{gluing in a replacement subdiagram}.  However,  providing a precise meaning for (i) will require care, and (ii) is also not straightforward because it jeodardises the diagram's planarity.  So here are  details.

	\begin{enumerate}[label=(\roman*)]

\item  \textbf{Cutting  a subdiagram  out of a diagram.}   Suppose $\Delta$ is a finite 2-complex embedded in the plane.  Let $\Gamma$ denote its 1-skeleton.   Our notion of \emph{cutting  a subdiagram  out of $\Delta$} will lead to \emph{subdiagrams} that map into $\Delta$ injectively on their interiors, but not necessarily on their boundaries.  Accordingly, an edge-path $\rho$ in $\Gamma$ around which we cut to extract a \emph{subdiagram} need not be injective; however, it must be \emph{non-crossing}.  

To be precise about what we mean, we look to the \emph{ribbon graph} $\hat{\Gamma}$, which is a planar representation of $\Gamma$ in which each vertex $v$ is replaced by a disc $D_v$ and each edge $e$ by a topological rectangle $R_e$.  As shown top in Figure~\ref{fig:non-crossing}, collapsing each disc to a vertex and each rectangle to an edge retracts $\hat{\Gamma} \onto \Gamma$.  We say that $\rho$ is \emph{non-crossing} when it lifts to a simple loop $\hat{\rho}$ in $\hat{\Gamma}$.

 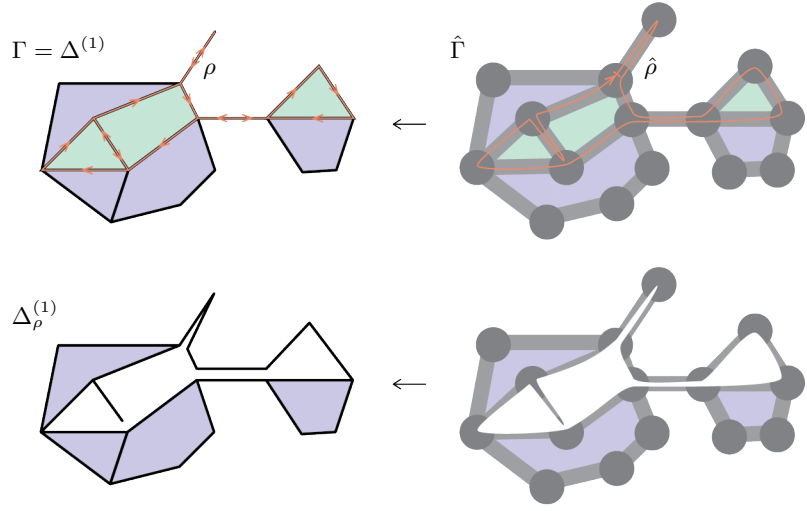
\begin{figure}[ht]
\begin{overpic}
{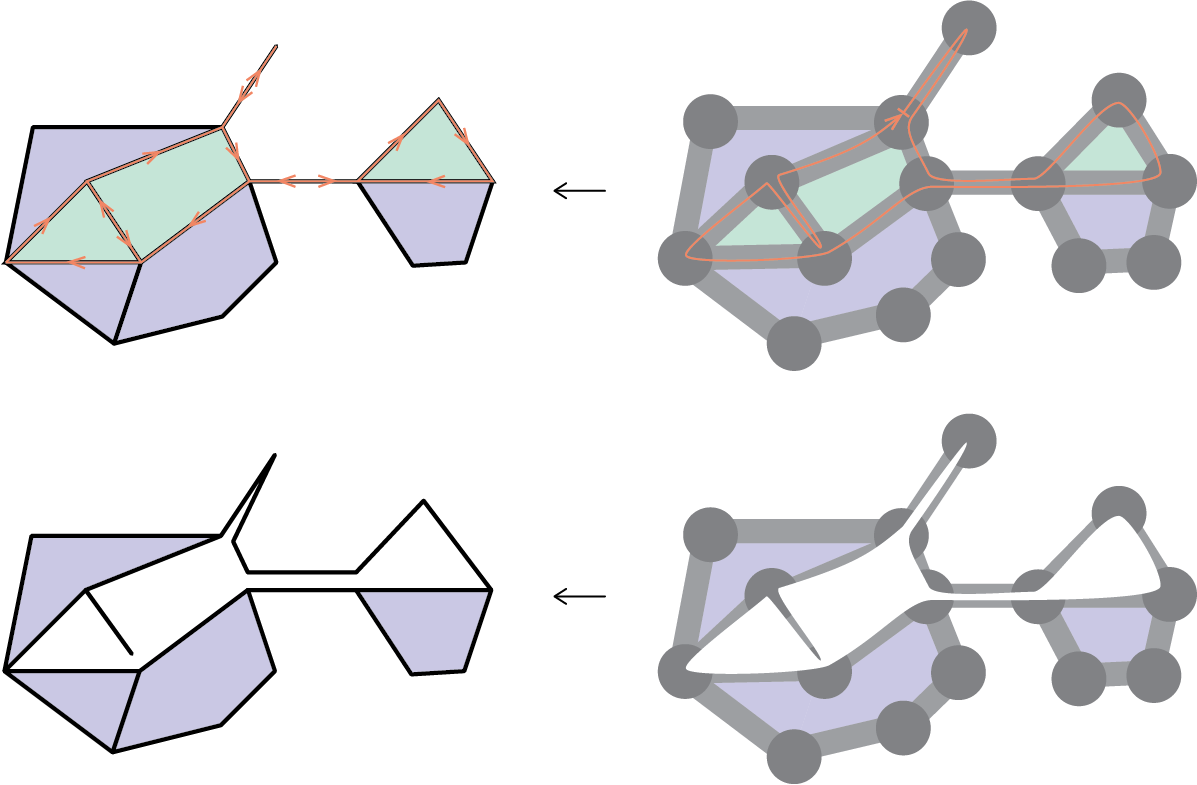}
      \put(62,162){\small{$\rho$}}  
      \put(228,162){\small{$\hat{\rho}$}}  
        \put(-10,168){\small{$\Gamma = \Delta^{(1)}$}}  
  \put(155,168){\small{$\hat{\Gamma}$}}  
        \put(-10,68){\small{$\Delta_{\rho}^{(1)}$}}  

\end{overpic}
 \caption{Cutting a subdiagram out of a diagram along a non-crossing edge-loop}
  \label{fig:non-crossing}
\end{figure}

Let $\text{int}( \hat{\rho} )$ denote the open subset of the plane enclosed by    $\hat{\rho}$.  The planar 2-complex $\Delta_{\rho}$ obtained by \emph{cutting $\Delta$ along $\rho$} is that to which $\hat{\Delta} \smallsetminus \text{int}( \hat{\rho} )$ retracts in the manner of the example in Figure~\ref{fig:non-crossing}  --- the connected components of $(\cup_v D_v) \smallsetminus \text{int}( \hat{\rho} )$  are mapped bijectively to the vertices of $\Delta_{\rho}$; the connected components of $(\cup_e R_e) \smallsetminus \text{int}( \hat{\rho} )$ are mapped bijectively to the edges  of $\Delta_{\rho}$; and the faces of $\Delta$ that are outside $\rho$ are mapped to  bijectively to the faces  of $\Delta_{\rho}$.

\item   \textbf{Sewing a van Kampen diagram into a hole.} Suppose $\Omega$ is an annular diagram
 whose outer and inner boundary cycles  are labeled by (non-empty) words $u$ and  $v$, respectively.   We define an operation, illustrated in Figure~\ref{fig:island}, for combining $\Omega$   with a van Kampen diagram $\Delta$ for $v$ to produce a van~Kampen diagram
 $\Omega \sqcup_v \Delta$ for $u$.

As per step I in the figure, place $\Delta$  within the hole in $\Omega$ and identify the appropriate vertex on the interior boundary of $\Omega$ with the appropriate vertex on  $\partial \Delta$, so as to give a planar annular diagram for the pair $u$, $v v^{-1}$.  Next, identify the edge $e_1$ labelled by the first letter, say $a$, of $v v^{-1}$ with the edge $e_2$ labelled by the last letter $a^{-1}$. As in our proof of Lemma~\ref{annular diagrams lemma}, the only circumstance is which this fails to preserve planarity is when $e_1$ and $e_2$ together make up the entire boundary of a subdiagram (for the word $aa^{-1}$).  In this event, we discard that subdiagram, leaving only the one or two vertices where it was attached to the rest of the diagram.   Repeat until the hole is closed up.  The result is 
 $\Omega \sqcup_v \Delta$

\end{enumerate}

  \begin{figure}[ht]
\begin{overpic}
{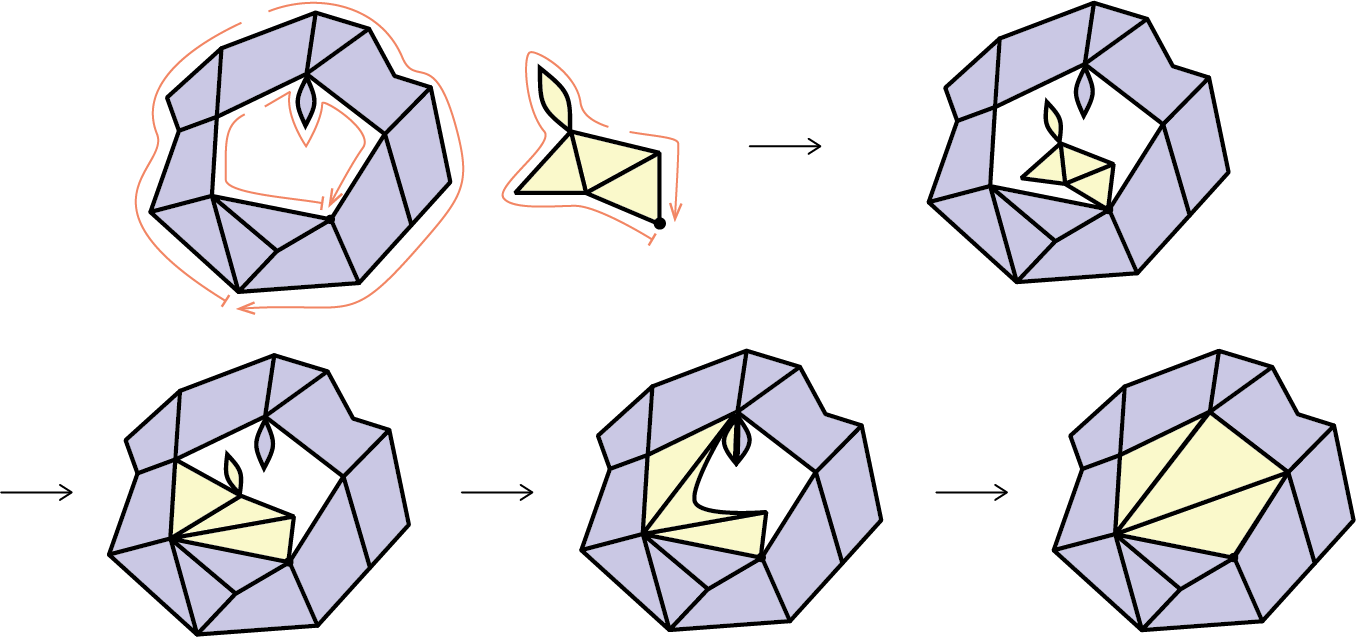}
      \put(59,148){\small{$u$}}  
      \put(60,124){\small{$v$}}  
      \put(147,121){\small{$v$}}  
 \put(100,143){\small{$\Omega$}}  
\put(155,130){\small{$\Delta$}}  
\put(318,60){\small{$\Omega \sqcup_v \Delta$}}  
\put(188,121){\small{I}}  
\put(7,38){\small{II}}  
\put(116,38){\small{III}}  
\put(230,38){\small{IV}}  
\end{overpic}
 \caption{An example illustrating sewing a van Kampen diagram $\Delta$  to fill the hole in an annular diagram $\Omega$.  In step I the  diagram $\Delta$ is inserted into the `hole' and attached at one vertex.  In step II two pairs of edges are folded together, and then two more pairs in step III. In step IV instead of folding the next pair, two 2-cells are discarded, and then a final two pairs are folded together.}
  \label{fig:island}
\end{figure}
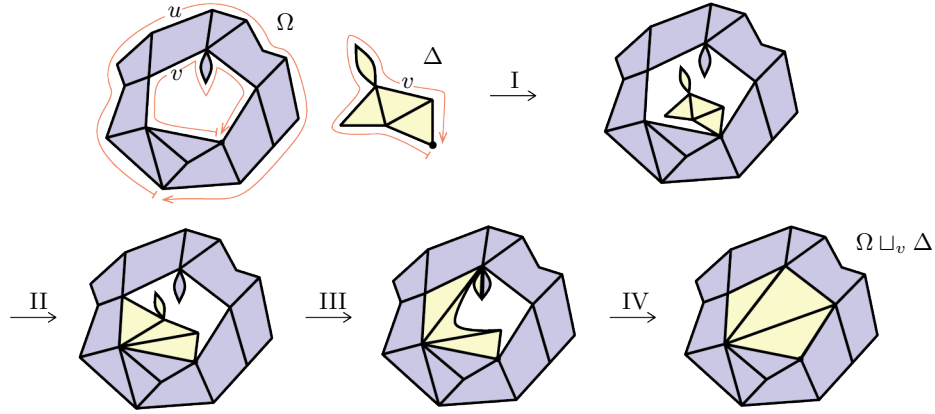

\begin{remark} \label{cut and replace remark}
We can combine operations (i) and (ii) to cut a subdiagram out of a van~Kampen diagram $\Delta$ for a word $u$ and glue in another van~Kampen diagram in its place to produce a new van~Kampen diagram for $u$.  Operations (i) and (ii) can also be performed as follows on an annular diagram $\Omega$ over a presentation $\<X\mid R \>$.  Let $u$ and $v$ be the words labelling the outer and inner boundary circuits, respectively, of $\Omega$.  Assume $u$ and $v$ do not represent the identity in $\<X\mid R \>$.  Suppose  $\rho$ is a non-crossing edge-loop that  is inessential (i.e., is contractible in $\Omega$). View $\Omega$ as a van~Kampen diagram over  $\<X\mid R\cup \{v\}\>$ by adding an additional face $f_v$
 filling the inner cycle of $\Omega$.  Then cutting out the subdiagram bounded by $\rho$ and sewing in another van~Kampen diagram in its place will produce a new van~Kampen diagram for $u$ over $\<X\mid R\cup \{v\}\>$.  As we saw, this sewing operation can remove faces, however the face $f_v$ must survive because otherwise we would arrive at a van~Kampen diagram for $u$ over $\<X\mid R \>$, which would mean that $u=1$ in $\<X\mid R \>$ contrary to hypothesis.  So 
  removing the interior of $f_v$ gives  an annular diagram for  $u$ and $v$ over   $\<X\mid R \>$.
\end{remark}

The following corollary sets out how  we can simplify the task of  estimating the  conjugator length functions of HNN-extensions. 
We will apply part \ref{item:3} to Stallings' group in our proof of Theorem~\ref{Stallings CL}.  This result is also called on in \cite{GiRi}.

\begin{cor} \label{cor:HNN simplify} 
Let $H$ be a group with finite presentation $\langle X \mid R \rangle$ and let $\phi: A \to B$ be an isomorphism between finitely generated subgroups of $H$.  This  defines an HNN-extension with stable letter $t$:
\begin{equation} \label{eq:pres hnn}
G \  = \  H \ast_\phi   \ = \  \left\langle X \cup \set{ t }  \, \left| \,  R \cup \set{ t^{-1} \alpha_i t =\beta_i  }_{i=1, \ldots, k} \right. \right\rangle
\end{equation}
 where $\alpha_1, \ldots, \alpha_k$ are words on $X$ that generate $A$ and $\beta_1, \ldots, \beta_k$ are words on $X$ that represent $\phi(\alpha_1), \ldots, \phi(\alpha_k)$, respectively.

	\begin{enumerate}[label=(\roman*)]
\item \label{item:1} If  a word $w$ represents the identity in $G$, then it admits a van~Kampen diagram over \eqref{eq:pres hnn} without $t$-annuli.  

\item \label{item:2} If $u$ and $v$ are words representing non-identity conjugate elements of $G$, then they admit an annular diagram $\Omega$ with no inessential $t$-annuli.  Further, $\Omega$  may have radial $t$-corridors or essential $t$-annuli, but cannot have both. Further, there exist words $u'$ and $v'$  that satisfy
	\begin{enumerate}[label=(\alph*)]
		\item \label{e:a} $u \sim  u'$ and $v \sim v'$ in $G$,
		\item  \label{e:b} $\CL_G(u,v) \leq \CL_G(u',v') + |u| + |v|$  
	\end{enumerate}		
and there is an annular diagram  $\Omega'$ for $u'\sim v'$ with the additional property that no $t$-corridor connects a pair of edges that are both in the $u'$-boundary cycle or both in the $v'$-boundary cycle. 
	\end{enumerate}		
Suppose, moreover, that $\alpha_i \equiv \beta_i$ (as words) for all $i$,  in which case $A=B$,  the stable letter $t$ commutes with elements of $A$,  and $G   =  H \dot{\ast}_{\phi}$. Then:  
	\begin{enumerate}[label=(\roman*)] \addtocounter{enumi}{2}
\item \label{item:3}  Further to \ref{item:2},  one can choose $u'$ and $v'$ so that  $|u'|\leq |u|$ and $|v'| \leq |v|$ and the  annular diagram $\Omega'$  has no essential annuli.  So there are no $t$-annuli in $\Omega'$ and all the $t$-corridors are radial.  If  $u'$  contains no  letter $t^{\pm 1}$, then neither does $v'$, 
and $u' \sim v'$ in $H$ with $\CL_G(u',v') = \CL_H(u',v')$.
	\end{enumerate}		
\end{cor}

\begin{proof}
For   \ref{item:1}, suppose $\Delta$ is a van Kampen diagram for such a $w$ over \eqref{eq:pres hnn}.  Suppose $\Delta$ has a $t$-annulus. The  word  $w_0$ around the outer boundary of an \emph{innermost} $t$-annulus in $\Delta$ represents $1$ in $H$, and so admits a van~Kampen diagram over  $\langle X \mid R \rangle$.  Per Remark~\ref{cut and replace remark} we can cut out this $t$-annulus and the subdiagram it encircles and replace it by a  subdiagram over  $P$, with the result being a van~Kampen diagram for $w$ with fewer $t$-annuli. The result follows by induction. 

For words $u$ and $v$ representing non-identity conjugate elements of $G$, Lemma~\ref{annular diagrams lemma} tells us that   there is an annular diagram $\Omega$ portraying $u\sim v$.  Arguing as above, per Remark~\ref{cut and replace remark} we can remove the inessential $t$-annuli in $\Omega$,  thereby   establishing the first part of  \ref{item:2}.  There cannot be  both a  radial $t$-corridor and an essential $t$-annulus in  $\Omega$ because they would have to cross, which cannot happen.

Towards the remaining part of \ref{item:2}, we appeal to the following process of successively eliminating `pinches' from  a word $w$  in $(A\cup\{t\})^{\pm 1}$. 
Let $w_0$ be $w$. For $i \geq 0$,  we recursively obtain $w_{i+1}$ from $w_i$ as follows: Let $\hat{w}_i$  be a cyclic permutation $\lambda_i t^{-\varepsilon_i} \mu_i t^{\varepsilon_i} \xi_i$ of $w_i$ (with $\varepsilon_i = \pm 1$) such that  $\mu_i$ is a word in $(A\cup\{t\})^{\pm 1}$ that represents an  element of $A$ if $\varepsilon_i = 1$  of $A$ and of $B$ if $\varepsilon_i = -1$;  let $w_{i+1} := \lambda_i  \overline{\mu_i}   \xi_i$,  where $\overline{\mu_i}$ denotes a word in $X^{\pm 1}$ representing  $\phi^{\varepsilon_i}(\mu_i)$. By construction,  $w_{i+1}= \hat{w}_i$ in $G$.  We continue until a word $w' = w_m$ is obtained whose cyclic permutations contain no pinches.

The successive moves in this process either \emph{nest} or are \emph{disjoint}, which is to say that $\overline{\mu_i}$ is either a subword of $\mu_{i+1}$ or has no overlap with it.  Indeed, the process can be depicted by an annular diagram of the form displayed in Figure~\ref{fig:arches} with each step giving rise to one $t$-corridor.
  It follows that if we replace $w$ by an appropriate cyclic permutation $\tilde{w}$, then we can execute this same process of removal of pinches without any subsequent cyclic permutations --- that is, $w_i = \hat{w}_i$ for all $i$.  Then, since only one 
  cyclic permutation was made,  $\CL_G(w,w')\le |w|$.

  \begin{figure}[ht]
\begin{overpic}
{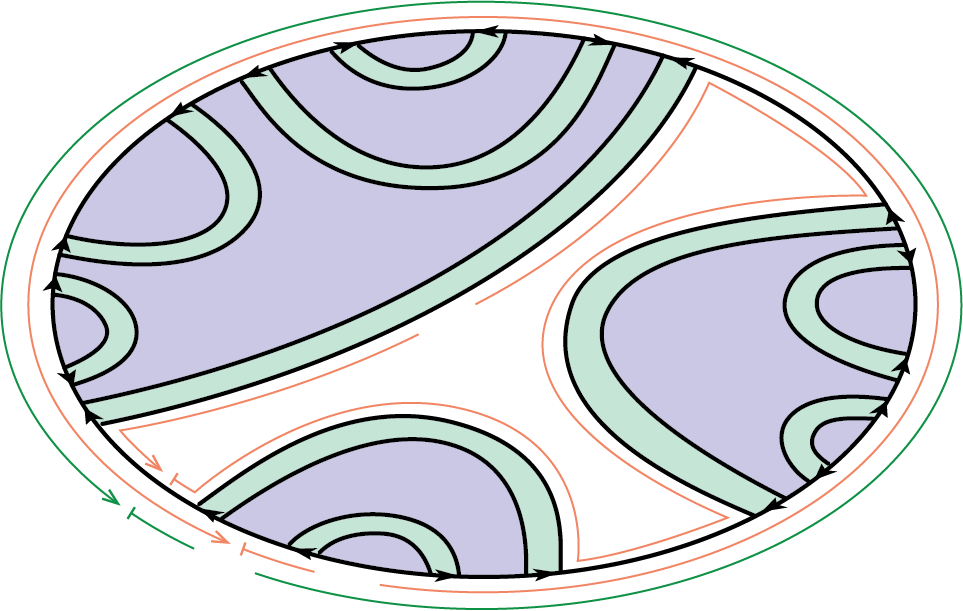}
\put(49,7){\small{$\tilde{w}$}}  
\put(104,66){\small{$w'$}}  
\put(80.5,5){\small{$w$}}  
\end{overpic}
 \caption{An illustration of $t$-corridors arising from a process of eliminating pinches in a word $w$.}
  \label{fig:arches}
\end{figure}

Now, for $u$ and $v$ of the lemma, define $u'$ and $v'$ to be words obtained in this manner.  Then $u' \sim u \sim v \sim v'$ and  conditions~\ref{e:a}  and \ref{e:b} are met, where for \ref{e:b} we are appealing to a form of triangle inequality enjoyed by conjugator length:
$$
\CL_G(u,v)\le \CL_G(u,u') +  \CL_G(u',v') +  \CL_G(v',v) \le |u| +  \CL_G(u',v')  + |v'|.
$$
As before,  we may assume there is  an annular diagram $\Omega'$ with no inessential $t$-annuli.  Such an $\Omega'$ can have no $t$-corridor that connects a pair of $t$-labelled edges that are both in the outer-boundary cycle (labelled $u'$) or both in the inner-boundary cycle  (labelled $v'$) --- else, a further pinch-removal would be possible.    
 
 For  \ref{item:3}, in this special case where $\phi$ is the identity map $A \to A$, the words $\overline{\mu_i}$  of the pinch-removal process are $\mu_i$, and so  $|u'|\leq |u|$ and $|v'| \leq |v|$.  
 And suppose there is an essential $t$-annulus in $\Omega'$.  The word read around its outer cycle $\rho$ is the same as that read around its inner cycle $\rho'$.  In this case, we can simplify the diagram by excising the corridor and folding the
 inner cycle of the deleted annulus to the outer one (per Remark~\ref{cut and replace remark}).  We continue until there are no essential $t$-annuli.

Finally,  if there are no letters $t^{\pm 1}$ in $u'$, then there are none in $v'$ since the only $t$-corridors in $\Omega'$ are radial.   Killing $t$ defines a retraction of free groups $r:F(A\cup\{t\})\to F(A)$ that induces a retraction of $G$ onto $H$.  
The words $u'$ and $v'$ are fixed by $r$, and if $u'w=wv'$ in $G$ then $r(w)$, which has length at most $|w|$, conjugates
$u'$ to $v'$ in $H$.  Thus $u'\sim v'$ in $H$ and $\CL_G(u',v')= \CL_H(u',v')$. 
\end{proof}

\section{Estimates of   conjugator-length functions --- a survey}  \label{survey}

With apologies for omissions, we present here a survey of known bounds  on conjugator length functions, exploring some  examples in more detail.  Where appropriate, statements of growth rates  should be understood to be up to $\preceq$, $\succeq$, or $\simeq$  as defined in Section~\ref{why}.

\subsection{Free, hyperbolic, and relatively hyperbolic groups} \label{sec:hyp groups}

Recall that, whereas we write $\CL(u, v)$ for the length of the shortest word $w$ such that $uw=wv$ in a finitely generated group $G$,  Antol\'in and Sale define  the permutation  conjugator length $\PCL(u, v)$  to be the  length of the shortest $w$ such that $u'w=wv'$ in $G$ for some cyclic permutations $u'$  of $u$ and $v'$ of $v$ (in the terminology of Section~\ref{sec: Notation and terminology}).

\begin{lemma} \label{lem:free groups}
 If  $G = F(A)$ is a free group on a finite generating set $A$, then its conjugator length function satisfies $\CL(n) \leq n$ for all $n$.  Indeed, words $u$ and $v$ satisfy $u \sim v$ in $G$  if and only if their cyclically reduced forms are cyclic permutations of one another.  And, in that event,   $\PCL(u, v) =0$.  
 \end{lemma}

Recall that  for  $\delta \geq 0$, a group with  finite generating set $A$ is \emph{$\delta$-hyperbolic} if every geodesic triangle in the  Cayley graph $\textup{Cay}(G,A)$ is \emph{$\delta$-slim}, which is to say that each side is
contained  in the $\delta$-neighbourhood of the union of the other two sides. For example,  the free group $F(A)$ is $0$-hyperbolic.  A path in a space is a  \emph{$k$-local geodesic} if every subpath of length at most $k$ is a geodesic.

The standing assumption for  the following four lemmas is that $\textup{Cay}(G,A)$ is $\delta$-hyperbolic.
Taken together, the first two  encapsulate much of the essence of hyperbolicity.

\begin{lemma}[via H.1.7 and H.1.13 of Chapter~III.$\Gamma$ \cite{BrH}]  \label{local close to geodesics}
Let $k=8 \delta +1$. There is a constant $R>0$ such that if $\rho$ is  $k$-local geodesic from $p$ to $q$ in $\textup{Cay}(G,A)$ and $[p,q]$ is a geodesic from $p$ to $q$, then $\rho$ lies  in the $R$-neighbourhood of $[p,q]$.   
\end{lemma}

\begin{lemma}[Proposition~1.6, page~400 of \cite{BrH}]  \label{geodesic close to path}
There exists $K>0$ such that, for all $L>1$,  if $\rho$ is a continuous path of length $L$ from $p$ to $q$ in $\textup{Cay}(G,A)$ and $[p,q]$ is a geodesic from $p$ to $q$, then $[p,q]$ lies in the $K \log L$-neighbourhood of $\rho$.   
\end{lemma}

\begin{lemma}[\textup{Lemma~2.9, page~452 of \cite{BrH}.}]   \label{geodesic quadrilateral lemma}
   Suppose $u \sim v$ in $G$ for some words $u$ and  $v$ that are \emph{fully reduced}, which is to say that all of their cyclic permutations are geodesic words.  Then  either (a) $\max \set{|u|, |v|} \leq 8 \delta +1$,  or (b) there exist  cyclic permutations $u'$ and $v'$ of $u$ and $v$
and a geodesic word $w$ of length $\PCL(u,v) \leq 2\delta +1$ so that $u'w=wv'$ in $G$.
 \end{lemma}

\begin{proof}[Proof outline] 
An equality  $u'w=wv'$ in $G$, 
where $u'$ and $v'$ are cyclic permutations of $u$ and $v$, respectively,  gives rise to a geodesic quadrilateral  $\mathcal{Q}$ in $\textup{Cay}(G,A)$ with sides labeled $u'$,  $w$,  ${v'}^{-1}$, ${w}^{-1}$.  
Inserting a geodesic path  as a diagonal subdivides $\mathcal{Q}$ into two geodesic triangles, and   the  $\delta$-slim condition tells us that for any point on the $u'$-side, there is a path of length at most $2\delta$ to one of the other three sides of $\mathcal{Q}$.  If we chose $u'$ and $v'$ so that $|w|$ is minimal, then the two $w$-sides of $\mathcal{Q}$ are of minimal length among all paths in $\textup{Cay}(G,A)$ that connect  a vertex on the $u'$-side to a vertex on the $v'$-side. We leave as an exercise for the reader the metric argument (explained in detail in \cite{BrH}) that now shows that either (a) and (b) holds.      
\end{proof}

If, instead of requiring  $u$ and $v$ be fully reduced, we only insist that they be geodesic words, we get:

\begin{lemma}[Section 3 of \cite{AS}]
  \label{geodesic hexagon lemma part}   
   Suppose $u$ and  $v$ are geodesic words such that $u \sim v$ in $G$.   Let $u'=u_2u_1$ and $v'=v_2v_1$ be cyclic permutations of $u=u_1 u_2$ and $v=v_1 v_2$ and let 
$w$ be a geodesic word  of length $\PCL(u,v)$ such that $u'w=wv'$ in $G$. Then for all $4 \delta < i < |w|-4 \delta$, 
the length-$i$ prefix $w_i$ of $w$ satisfies $d(1, w_i^{-1} u' w_i)   \leq 8 \delta$. 
 \end{lemma}

\begin{proof} 
The equality $u_2u_1 w = w v_2v_1$ in $G$ between six geodesic words gives rise to a geodesic hexagon $\mathcal{H}$ in $\textup{Cay}(G,A)$, as shown left in Figures~\ref{fig:hexagons} and \ref{fig:more_hexagons}.  	Suppose  $4 \delta < i < |w|-4 \delta$.      Write $w = w_i \hat{w}_i$ where $w_i$ is the length-$i$ prefix of $w$ and $\hat{w}_i$ is the remainder of $w$.   Let $p$ be the vertex of $\mathcal{H}$ at the end of one copy of $w_i$, as shown in both figures. 

On subdividing    	  $\mathcal{H}$ into four $\delta$-slim geodesics by inserting three diagonals, we find that $\delta$-hyperbolicity implies that each side of $\mathcal{H}$ is in the $4 \delta$-neighbourhood of the other five sides. So there is a point $y$, indeed a vertex $y$, on one of the other five sides such that $d(x,y) \leq 4 \delta$ in $\textup{Cay}(G,A)$.  Let $\sigma$ be the word read along some geodesic in $\textup{Cay}(G,A)$ from $y$ to $x$.

  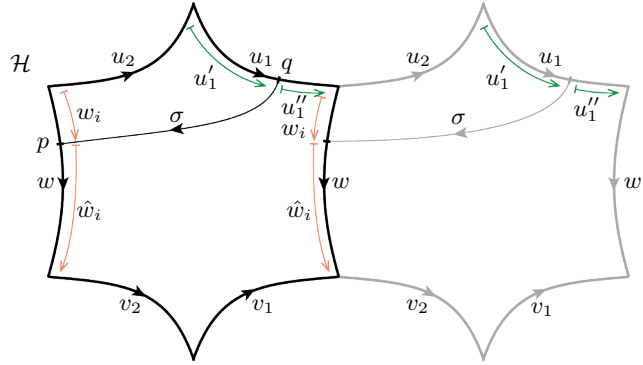
\begin{figure}[ht]
\begin{overpic}
{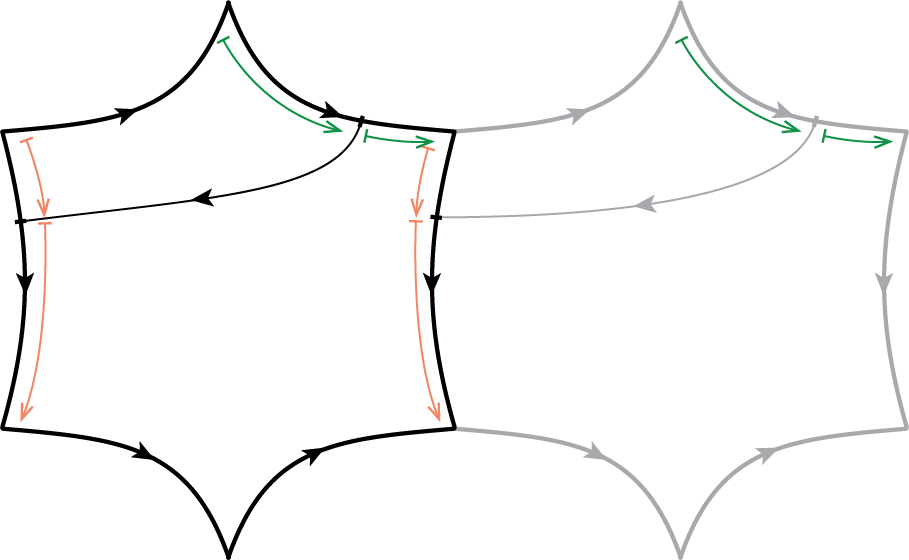}
\put(11,91){\small{$w_i$}}  
\put(11,54){\small{$\hat{w}_i$}}  
\put(88,85){\small{$w_i$}}  
\put(90,54){\small{$\hat{w}_i$}}  
\put(-4,65){\small{$w$}}  
\put(107,65){\small{$w$}}  
\put(217,65){\small{$w$}}  
\put(46,90){\small{$\sigma$}}  
\put(152,89){\small{$\sigma$}}  
\put(-4,80){\small{$p$}}  
\put(88,110){\small{$q$}}  
\put(76,112){\small{$u_1$}}  
\put(24,112){\small{$u_2$}}  
\put(77,19){\small{$v_1$}}  
\put(27,19){\small{$v_2$}}  
\put(55,104){\small{$u'_1$}}  
\put(88,93){\small{$u''_1$}}  
\put(185,112){\small{$u_1$}}  
\put(135,112){\small{$u_2$}}  
\put(135,19){\small{$v_2$}}  
\put(182,18){\small{$v_1$}}  
\put(165,104){\small{$u'_1$}}  
\put(198,92){\small{$u''_1$}}  
\put(-14,109){\small{$\mathcal{H}$}}  
\end{overpic}
 \caption{Geodesic hexagons as arising in the proof of Lemma~\ref{geodesic hexagon lemma part} in the case where $q$ is on the side of $\mathcal{H}$ labelled by $u_1$.}
  \label{fig:hexagons}
\end{figure}

Suppose $q$ is on the side of $\mathcal{H}$ labelled by $u_1$, which is the situation depicted in the figure. 
Express $u_1$ as the concatenation $u'_1 u''_2$ of subwords, with the transition happening at $q$.    	  
	  
  The figures depict a second hexagon, the translation of $\mathcal{H}$ that shares one $w$-side with $\mathcal{H}$ as shown.  Tracing paths through Figure~\ref{fig:hexagons} starting from $q$, we see that $w' = \sigma \hat{w}_i$ satisfies  $$u''_1 u_2 u'_1 w' = w' v_2 v_1$$ in $G$.  Now, $u''_1 u_2 u'_1$ and $v_2v_1$ are cyclic permutations of $u$ and $v$, respectively, so $|w| = \PCL(u,v)$  forces  $|w'| \geq |w|$,  but that is impossible given that $|w'| = |\sigma| + | \hat{w}_i |$,  $|w| =	|w_i| + | \hat{w}_i |$, $|\sigma| \leq 4 \delta$, and $|w_i| = i > 4 \delta$.    
	  
So $q$ is not on the side of $\mathcal{H}$ labelled by $u_1$.  Similarly, $q$  cannot be on the sides labelled by $u_2$, $v_1$, or $v_3$.  

So in fact $q$ is on the second side of  $\mathcal{H}$ that is labelled by $w$, as shown in Figure~\ref{fig:more_hexagons}.  Express $w$ as the concatenation $w_j \hat{w}_j$ of subwords, with the transition happening at $q$.  Let $r$ be the vertex at the end of $w_i$ on that same side as $q$.  

Suppose $i \leq j$, which is the situation depicted in the figure.  So $d(q,r) = j-i$ because $w$ is a geodesic word.

  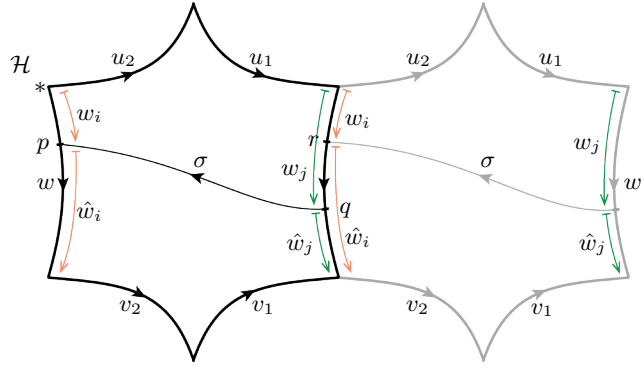
\begin{figure}[ht]
\begin{overpic}
{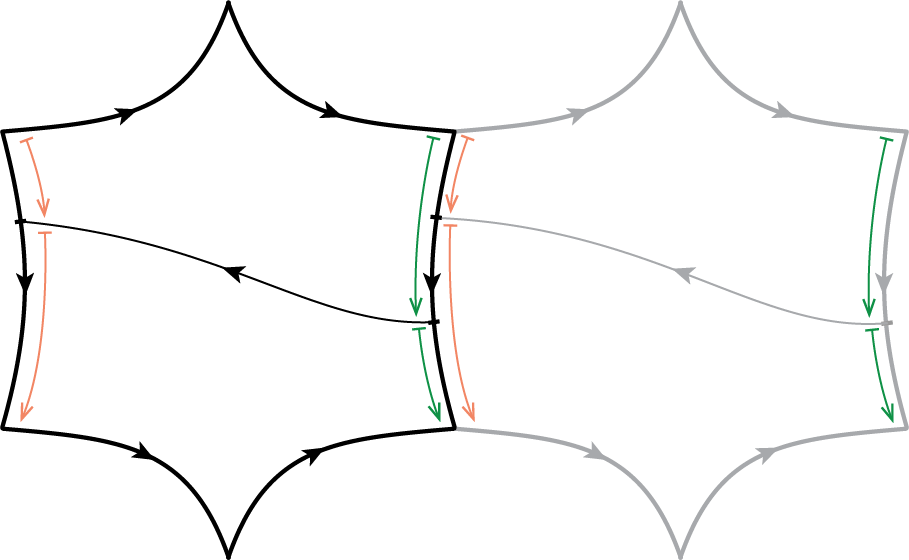}
\put(11,91){\small{$w_i$}}  
\put(11,54){\small{$\hat{w}_i$}}  
\put(88,70){\small{$w_j$}}  
\put(90,42){\small{$\hat{w}_j$}}  
\put(112,89){\small{$w_i$}}  
\put(112,44){\small{$\hat{w}_i$}}  
\put(198,80){\small{$w_j$}}  
\put(200,42){\small{$\hat{w}_j$}}  
\put(-4,65){\small{$w$}}  
\put(217,65){\small{$w$}}  
\put(55,73){\small{$\sigma$}}  
\put(163,73){\small{$\sigma$}}  
\put(-4,80){\small{$p$}}  
\put(110,56){\small{$q$}}  
\put(98,81){\small{$r$}}  
\put(76,112){\small{$u_1$}}  
\put(24,112){\small{$u_2$}}  
\put(77,19){\small{$v_1$}}  
\put(27,19){\small{$v_2$}}  
\put(185,112){\small{$u_1$}}  
\put(135,112){\small{$u_2$}}  
\put(135,19){\small{$v_2$}}  
\put(182,18){\small{$v_1$}}  
\put(-14,109){\small{$\mathcal{H}$}}  
\put(-5,101){\small{$\ast$}}  
\end{overpic}
 \caption{Geodesic hexagons as arising in the proof of Lemma~\ref{geodesic hexagon lemma part} in the case where $p$ and $q$ are on the two sides of $\mathcal{H}$ labelled by $w$.}
  \label{fig:more_hexagons}
\end{figure}		  
	  
Let $w' = w_i \sigma^{-1}  \hat{w}_j$. Tracing paths through Figure~\ref{fig:more_hexagons} starting from the vertex labelled $\ast$ reveals that $u' w' = w' v'$ in $G$.  	     
Given that $|\sigma| \leq 4 \delta$, it must be the case that $d(q,r) \leq 4 \delta$, because otherwise $w'$ would be a shorter conjugator than $w$.  So $$d(1, w_i^{-1} u' w_i) = d(p,r) \leq d(p,q) + d(q,r)    \leq 4 \delta + 4 \delta  =  8 \delta,$$ as required.  

If, instead, $j <i$, then $w' = w_j \sigma \hat{w}_i$ would satisfy $u'w' = w' v'$ in $G$, and the same argument applies.    	  
\end{proof}

We are now ready to present the analog of Lemma~\ref{lem:free groups} for hyperbolic groups. 

\begin{thm}[\cite{Lysenok},  \cite{BrH} pp.\ 451--454, \cite{AS}] \label{thm: CL for hyperbolic groups}  For all $\delta \geq 0$, there exists  $C>0$ such that if $G$ is a $\delta$-hyperbolic group, then:

\begin{enumerate}[label=(\roman*)]
\item  \label{hyp linear bound}	 The  conjugator length function of $G$ satisfies $\CL(n) \leq C n$ for all $n$.  
\end{enumerate}

For words  $u$ and $v$ of total length $n$ such that $u \sim v$ in $G$:

\begin{enumerate} [label=(\roman*)]
\addtocounter{enumi}{1}
\item   \label{BrH hypothesis} If all cyclic permutations of $u$ and $v$ are  $(8 \delta +1)$-local geodesics, then $\PCL(u, v) \leq  C$. 

\item  \label{AS hypothesis}   If  $u$ and $v$ are  geodesic words, then $\PCL(u, v) \leq  C$.

\item \label{log} In general,  $\PCL(u, v) \leq  C \log n$.

\end{enumerate}
\end{thm}

\begin{figure}[ht]
\begin{overpic}
{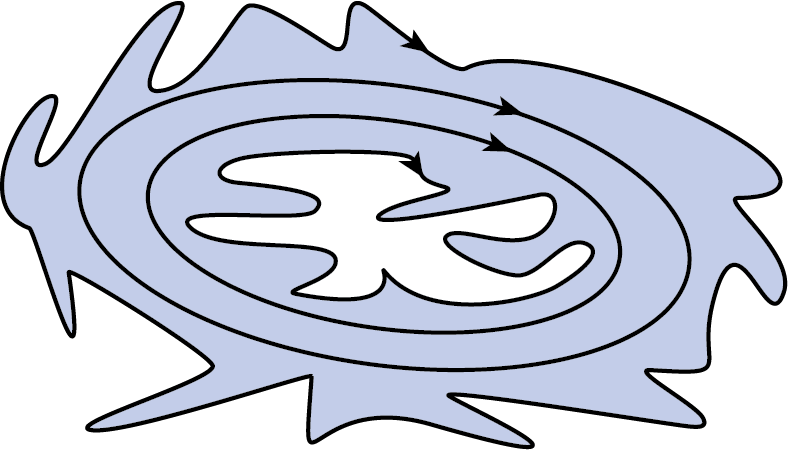}
 \put(102,99){\small{$u$}}     
 \put(93,63){\small{$v$}}     
 \put(123,84){\small{$\widetilde{u}$}}     
 \put(117,63){\small{$\widetilde{v}$}}   
\end{overpic}
 \caption{An annular diagram in a hyperbolic group}
  \label{fig:hyperbolic_annulus}
\end{figure}
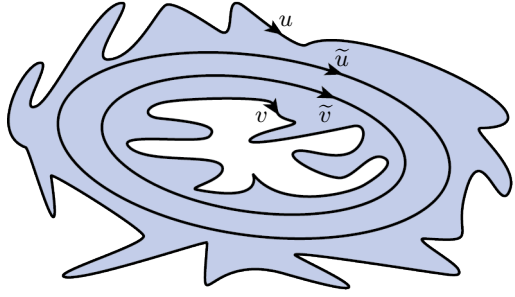

\begin{proof} 
Conclusions \ref{BrH hypothesis}--\ref{log} strengthen \ref{hyp linear bound}.  Indeed, \ref{hyp linear bound} is an immediate consequence of  \ref{log}. It is worth further noting that  deducing \ref{hyp linear bound} from \ref{BrH hypothesis} in the following manner leads to an efficient algorithmic solution to the conjugacy problem.  We can assume we have a finite presentation for $G$ that includes all words of length at most $16\delta +1$ that represent $1$ in $G$  among the defining relators.   Convert $u$ (respectively, $v$)  to a word $\widetilde{u}$ (respectively, $\widetilde{v}$) that represents a conjugate of $u$ (respectively, $v$) in $G$ so that $\widetilde{u}$  and $\widetilde{v}$  and their cyclic permutations are $(8 \delta +1)$-local geodesics; this can be done by a combination of cyclic permutation and at most $n$ applications of defining relators, each instance of the latter replacing a non-geodesic subword $\sigma$ of length at most $(8 \delta +1)$ by a geodesic word that equals $\sigma$ in $G$. So the pairs $u, \widetilde{u}$ and $v, \widetilde{v}$ admit annular diagrams (see Figure~\ref{fig:hyperbolic_annulus}) with no more than $n$ faces in total, and therefore no more than $n + n(16\delta +1)$ edges in total. So it follows from Lemma~\ref{lem:annular diagram path} that for  $C = 16\delta +2$, we have $\CL(u, \widetilde{u}) + \CL(v, \widetilde{v}) \leq Cn$ for all $n$.        The existence of $C>0$ such that  $\CL(n) \leq C n$ for all $n$ then follows because  $\CL(u, v) \leq  \CL(u, \widetilde{u}) +  \CL( \widetilde{u}, \widetilde{v}) + \CL( \widetilde{v}, v)$.

 	Here is a proof of  \ref{BrH hypothesis}.  Assume first that all cyclic permutations of $u$ and of $v$ are geodesic words.
 	  So Lemma~\ref{geodesic quadrilateral lemma}  applies. Under its case \textup{(i)},   $\CL( u, v)$ is at most a constant since it can only apply to some finite set of pairs $u$ and $v$.  And under its case \textup{(ii)}, we have  $\PCL( u, v) \leq  2 \delta +1$.  
 	 	Lemma~\ref{local close to geodesics} allows us to get the same result, with appropriately adjusted  the constants, when we only require that all cyclic permutations of $u$ and of $v$ be  $(8 \delta +1)$-local geodesics.   
 	
 	For  \ref{AS hypothesis}, we have that $u' w = w v'$ in $G$ for some $w$ such that $\PCL( u, v) = |w|$ and some cyclic permutations $u'$ and $v'$ of $u$ and $v$, respectively.   
 	 	Suppose the length-$i$ and length-$j$ prefixes $w_i$ and $w_j$ of $w$ satisfy $w_i^{-1} u' w_i = w_j^{-1} u' w_j$ in $G$ for some  $i < j$.  Then $w_i w_j^{-1} w$ would freely reduce to a word $w_0$ that is strictly shorter than $w$ and satisfies $u' w_0 = w_0 v'$ in $G$, which would contradict $\PCL( u, v) = |w|$.   Lemma~\ref{geodesic hexagon lemma part} tells us that if $4 \delta < i < j < |w|-4 \delta$, then    $d(1, w_i^{-1} u' w_i)$ and $d(1, w_j^{-1} u' w_j)$ are at most $8 \delta$ for all $i$ and $j$.  So the  set of $i$ such that $4 \delta < i  < |w|-4 \delta$ has size at most the number $C'$ of group elements in the ball of radius $8 \delta$ about $1\in G$. So if we take  $C = C' + 8 \delta$, then $|w | \leq C$, giving  \ref{AS hypothesis}.

For \ref{log}, we apply  \ref{AS hypothesis} to  geodesic words  $\bar{u}$ and $\bar{v}$ that represent the same elements of $G$ as $u$ and $v$, respectively.  We get that $\PCL(\bar{u}, \bar{v}) \leq  C$ and we then use Lemma~\ref{local close to geodesics} to deduce $\PCL(u, v) \leq  C \log n$  after a few adjustment of the constant $C$.
 	\end{proof}

For groups $G$ hyperbolic relative to parabolic subgroups $H_\omega$, for $\omega \in \Omega$, Antol\'{i}n  \& Sale \cite{AS} prove that $\CL_G(n) \simeq \max\{\CL_{H_\omega}(n)  : \omega\in\Omega\}+n$.  This  extends prior results of Ji, Ogle and Ramsey \cite{JOR} and complements Bumagin's study of the complexity of the conjugacy problem \cite{Bumagin1,Bumagin2}.

Information is available about the constants appearing in conjugator-length estimates of hyperbolic groups.  Bridson and Howie \cite[Proposition 2.3]{BH} proved that the conjugator length function of a group that is $\delta$-hyperbolic with respect to a $k$-element generating set satisfies  $\CL(n) \leq   n + (2k+1)^{4 \delta}+4 \delta$.  And recently, Wang and  Zhang \cite{WZ} showed that    surface groups $\pi_1(\Sigma_g) = \langle c_1, \ldots, c_{2g} \mid c_1 \cdots c_{2g} c_1^{-1} \cdots c^{-1}_{2g}  \rangle$ for which the genus $g$ is at least $2$  have  $\CL(2n) = \CL(2n + 1) \leq  n + 8g - 1$.

\subsection{Non-positively curved groups} \label{NPC section}

The general solution to the conjugacy problems in biautomatic, $\CAT(0)$, or more generally semihyperbolic groups \cite{AlonsoBridson, BrH, Epstein, GS} leads to an exponential upper bound on conjugator length as follows.  

For $k \geq 0$, a \emph{$k$-synchronous bicombing} for a group $G$ with finite generating set $A$ is a set of words $\set{w_g}_{g \in G}$ on $A^{\pm 1}$ such that $w_g$ represents $g$ in $G$  and the  $k$-synchronous fellow-traveler property (f.t.p.) holds: if we regard a word $w_g$ as a unit-speed path in the Cayley graph $ \textup{Cay}(G,A)$
starting at the identity (so $w_g(t)$ is the group element represented by the length-$t$ prefix of $w_g$), then   
$$d(w_g(t), a w_{a^{-1} g a'}(t)) \  \leq \ k$$  for all $t \in \N$, all $a, a' \in A^{\pm 1} \cup \set{1}$, and all $g \in G$.   This is illustrated in Figure~\ref{fig:bicombing}.

\begin{figure}[ht]
\begin{overpic}
{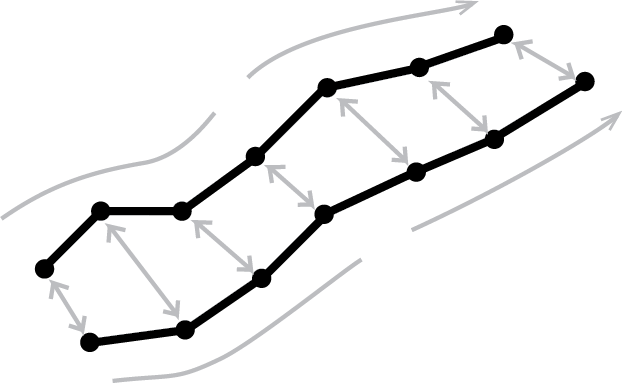}
 \put(13,5){\small{$1$}}     
 \put(2,25){\small{$a$}}     
 \put(145,73){\small{$g$}}     
 \put(124,87){\small{$ga'$}}     
\put(89,32){\small{$w_g$}}     
\put(36,69){\small{$aw_{a^{-1}ga'}$}}     
     \put(0,15){\tiny{$\leq k$}}     
 \put(22,23){\tiny{$\leq k$}}     
 \put(41,29){\tiny{$\leq k$}}     
 \put(58,43){\tiny{$\leq k$}}     
 \put(77,57){\tiny{$\leq k$}}     
 \put(99,62){\tiny{$\leq k$}}     
 \put(120,73){\tiny{$\leq k$}}     
\end{overpic}
 \caption{The synchronous bi-combing condition}
  \label{fig:bicombing}
\end{figure}
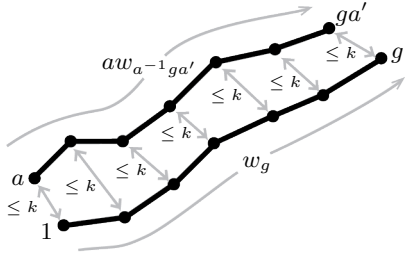

This property is an essential feature of \emph{non-positive curvature} in groups.  It is enjoyed (for some $k \geq 0$) by hyperbolic groups:  $w_g$ can be taken to be any geodesic representative of $g$,  and the fellow-traveler property (f.t.p.)  follows from the thinness of  geodesic quadrilaterals in  $ \textup{Cay}(G,A)$.
The f.t.p.\ is also enjoyed by  $\CAT(0)$-groups, i.e.~any group that acts properly and cocompactly by isometries on a $\CAT(0)$-space $X$.  In this case, the words $w_g$ are obtained by approximating geodesics in $X$, and the f.t.p.\ is a coarse reflection of the convexity of the metric in $X$; see \cite[Lemma 1.10, page 445]{BH}.   In more detail,   one fixes a basepoint $x\in X$ and maps $G$ to $X$ by the quasi-isometry $g\mapsto g.x$, then extends this to a $G$-equivariant map  $ \textup{Cay}(G,A)\to X$ that sends edges to geodesics; one then defines $w_g$ to be the label on a shortest edge-path in $ \textup{Cay}(G,A)$ such that  the unique geodesic $[x,g.x]$ stays in the $C$-neighbourhood of the edge-path, where $C$ is a constant that depends on the constants of the quasi-isometric embedding. 

The f.t.p.\ holds for synchronously biautomatic groups \cite{Epstein} and for seimhyperbolic groups \cite{AlonsoBridson} by definition.

\begin{thm}[Cf.\ Chapter III.$\Gamma$ of \cite{BH}] \label{thm: CL for bicombable groups}  If  a group $G$  group  with finite generating set $A$  admits a $k$-synchronous bicombing $\set{w_g}_{g \in G}$, then it is finitely presentable and its conjugator length function satisfies $\CL(n) \preceq 2^n$.
	
\end{thm}

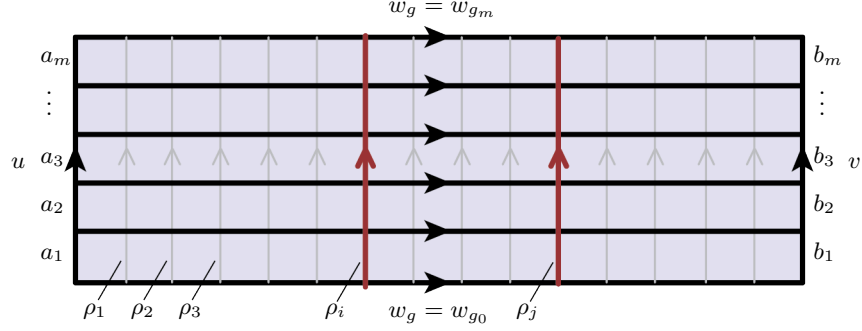
\begin{figure}[ht]
\begin{overpic}
{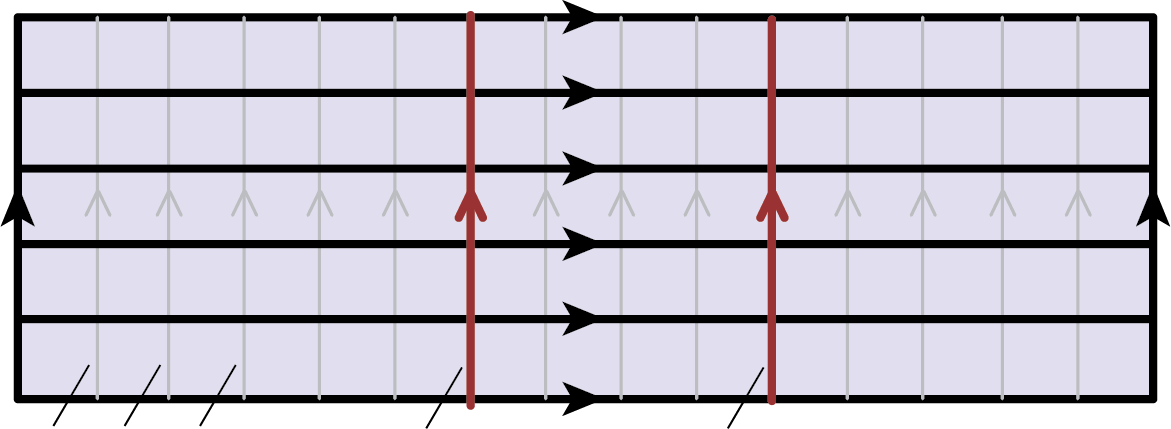}
 \put(-20,50){\small{$u$}}     
 \put(-9,16){\small{$a_1$}}     
 \put(-9,34){\small{$a_2$}}     
 \put(-9,52){\small{$a_3$}}     
 \put(-7,70){\small{$\vdots$}}     
 \put(-9,90){\small{$a_m$}}     
 \put(295,50){\small{$v$}}     
  \put(282,16){\small{$b_1$}}     
 \put(282,34){\small{$b_2$}}     
 \put(282,52){\small{$b_3$}}     
 \put(284,70){\small{$\vdots$}}     
 \put(282,90){\small{$b_m$}}  
 \put(122,-6){\small{$w_g = w_{g_0}$}}     
 \put(122,109){\small{$w_g = w_{g_m}$}}     
\put(7,-5){\small{$\rho_1$}}     
\put(25,-5){\small{$\rho_2$}}     
\put(43,-5){\small{$\rho_3$}}     
\put(98,-5){\small{$\rho_i$}}     
\put(170,-5){\small{$\rho_j$}}     

\end{overpic}
 \caption{Conjugacy diagram arising from a bicombing}
  \label{fig:bicombing_annulus}
\end{figure}

\begin{proof}  
Suppose $u$ and $v$ are words of total length $n$ that represent conjugate elements of  $G$.  Then $ug=gv$ in $G$ for some $g \in G$. So $u w_g=w_g v$ in $G$ per the perimeter of the diagram shown in  Figure~\ref{fig:bicombing_annulus}; we will
detail the construction in a manner that explains the label preserving map from its 1-skeleton to the Cayley graph $\textup{Cay}(G,A)$.      

For a word $\sigma$ and for $j \in  \N$, let  $\sigma(j)$ be the length-$j$ prefix of $\sigma$, with the understanding that $\sigma(j)=\sigma$ when $j \geq |\sigma|$.

Write $u=a_1 \cdots a_{|u|}$ and $v=b_1 \cdots b_{|v|}$ where $a_1, \ldots, a_{|u|}, b_1, \ldots, b_{|v|} \in A^{\pm 1}$ and  $a_{|u|+1}$, $a_{|u|+2}$, \ldots and $b_{|v|+1}$, $b_{|v|+2}$, \ldots are empty words. Let $m =  \max \set{|u|, |v|}$.  For $i \in \N$, let  $g_i = u(i)^{-1} w_g v(i)$.  Then  \begin{equation}  w_{g_i}b_{i+1} \  = \  a_{i+1}w_{g_{i+1}}. \label{rectangle boundary relation} \end{equation}  If we view Figure~\ref{fig:bicombing_annulus} as a stack of unit-high rectangles, then  $w_{g_0}, \ldots, w_{g_m}$ are the words read along the horizontal lines between the rectangles and \eqref{rectangle boundary relation} equates two ways of traversing the perimeter of those rectangles from  the lower-left corner to the upper-right.

  Given \eqref{rectangle boundary relation}, the synchronous bi-combing condition tells us that for all $0 \leq i  \leq m-1$ and $j  \in \N$, there is a word of length at most $k$ on $A$ that equals $w_{g_i}(j)^{-1}a_{i+1} w_{g_{i+1}}(j)$ in $G$.  So we can subdivide the rectangles, as shown schematically in Figure~\ref{fig:bicombing_annulus}, by inserting vertical paths  of length at most $k$  from the  vertex where $w_{g_i}(j)$  ends to the vertex where  $w_{g_{i+1}}(j)$ ends for all $j$. 
 Considering this sequence of vertical paths, we deduce that the \emph{quasi-monotone conjugacy property} (\emph{q.m.c.p.}) of \cite{BrH}  holds, which is to say: there exists a constant $K$ (in this case $K=k$), such that if $u \sim v$ in $G$, then there exists $w$ (in this case $w=w_g$) such that 
 for all $j$, there is a   word  $\rho_j$ of length at most $K \max \set{|u|, |v|}$ that equals $w(j)^{-1} u w(j)$ in $G$.    

 Let $R$ be the set of all words on $A^{\pm 1}$ of length at most $2 + 2k$ that represent $1$ in $G$. Since  the perimeters of the small squares that fill Figure~\ref{fig:bicombing_annulus} are all at most  $2 + 2k$, it is a van~Kampen diagram for $w_g^{-1} u w_g v^{-1}$ over the finite presentation  $\mathcal{P} = \langle A \mid R \rangle$.  Indeed, from van~Kampen's lemma (Lemma~\ref{vKs lemma}) and the following degenerate case of this construction, we learn that $\mathcal{P}$ is a finite presentation for $G$: if $u$ represents $1$, $v$ is the empty word and $g=1$, then we have a van~Kampen diagram for $u$ with respect to $\mathcal{P}$.

Finally, we deduce    $\CL(n) \preceq 2^n$ from the  q.m.c.p..  Given $u$ and $v$ such that $u \sim v$ in $G$, suppose $w$ is a minimal length word certifying the q.m.c.p.\  for the pair $u$ and $v$.   Suppose $\rho_i$ and $\rho_j$ are the same word for some $0  \leq i < j \leq |w|$.  Let $\widehat{w}$ be $w$ with its letters in positions $i+1, \ldots, j$ removed. Then $\widehat{w}$ would also certify the q.m.c.p.\ for $u$ and $v$, contradicting minimality of $|w|$. (Indeed, we could remove the portion  between them from the diagram of Figure~\ref{fig:bicombing_annulus}, and identify so as to get a diagram for $\widehat{w}^{-1} u \widehat{w} v^{-1}$.)  There are at most $(2|A| +1)^{km}$ words of length at most $km$ and $m \leq n$, so the result follows.   
\end{proof}

In light of Theorem~\ref{thm: CL for bicombable groups} we ask:

\begin{Open problem}
What are the optimal upper bounds in general on the conjugator length functions of    biautomatic  groups and of $\textup{CAT}(0)$ groups?  
\end{Open problem}

Kokarev~\cite{Kokarev} gives bounds, but they vary with the conjugacy class.

Relaxing the definition  of synchronous bicombing by only insisting it apply when $a=1$, gives the  class of  \emph{synchronously combable groups}.  There exist synchronously combable groups for which the conjugacy problem is unsolvable \cite{Bridson9}.  These examples have combings with quadratic length functions and hence admit cubic upper bounds on their Dehn functions.  Earlier, Miller \cite[Section~7]{Miller} had explained how it follows from results in \cite{BGSS, Miller1} that there are asynchronously automatic groups with unsolvable conjugacy problem.  These are of  the form $F_r \rtimes F_5$ and have cubic Dehn functions.
 Bridson's and Miller's examples show that there is no recursive upper bound on the  conjugator length functions of these two classes of groups.  

The groups in these two families are not synchronously automatic,  so do not shed light on the issue of what the optimal upper bound  is on the conjugator length functions of synchronously automatic groups.  But this question is rather premature, since  it remains unknown whether automatic groups can have unsolvable conjugacy problem \cite[Question 2.5.8]{Epstein}.

\subsection{Nilpotent groups} \label{sec: Heisenberg}

The 3-dimensional integral Heisenberg group  
$$\mathcal{H}_3(\Z)  \ = \  \set{ \  \left. \begin{pmatrix}
1  & \beta & \gamma  \\
  & 1 & \alpha \\
  & & 1
\end{pmatrix} \  \right| \ \alpha, \beta, \gamma \in \Z \   }$$
presented by
 \begin{equation} 
H \ =   \  \left\langle \; a,b,c \; \left| \; [a,b]c^{-1}, \; [a,c], \; [b,c]    \; \right.  \right\rangle\label{pres Heisenberg}
\end{equation}
 is a class-2 nilpotent group that serves as a fundamental example for the study of conjugator length.   
Every element has  a unique representative of the form $a^{\alpha} b^{\beta} c^{\gamma}$ with $\alpha, \beta, \gamma \in \Z$.  
 
\begin{thm} \label{Heisenberg group thm}
The conjugator length function of the 3-dimensional integral Heisenberg group grows quadratically.  
\end{thm}

\begin{proof} 
We begin by establishing the upper bound.  Our proof   hinges on the following lemma (which we will call on again in our study of Stallings' group in Section~\ref{Stallings section}).

\begin{lemma}[e.g.\ \cite{BFRT, BrRi2, Kornhauser}] \label{linear Diophantine equation}
	If the Diophantine equation $Ax + By =C$, where $A, B, C \in \Z$,  has a solution with $x, y \in  \Z$, then it has such a solution with  
	\begin{equation}
	  |x|,|y|   \ \leq \   \max\set{|A|, |B|, |C|}.   
\end{equation}  
\end{lemma}

Continuing with the proof of Theorem~\ref{Heisenberg group thm}, 
 for elements 
$$U = \begin{pmatrix}
1  & u & p  \\
  & 1 & \hat{u} \\
  & & 1
\end{pmatrix}, \quad 
V = \begin{pmatrix}
1  & v & q  \\
  & 1 & \hat{v} \\
  & & 1
\end{pmatrix}, \quad 
W = \begin{pmatrix}
1  & x & r  \\
  & 1 & \hat{x} \\
  & & 1
\end{pmatrix}$$ of $H$, the condition $UW = WV$ amounts to 
$$\begin{pmatrix}
1  & u+ x & r+u \hat{x}+p  \\
  & 1 & \hat{u} + \hat{x} \\
  & & 1
\end{pmatrix} \ =  \ \begin{pmatrix}
1  & v+ x & q +\hat{v} x+ r  \\
  & 1 & \hat{v} + \hat{x} \\
  & & 1
\end{pmatrix}.$$
So, given such $U$ and $V$ such that $U \sim V$ in $H$, it must be that $u=v$ and $\hat{u}=\hat{v}$, and the equation 
\begin{equation} \label{the pivotal Diophantine eqn}
	 \hat{u} x  - u \hat{x}  \ = \ p  - q,   
\end{equation}    
in which $x$ and $\hat{x}$ are viewed as variables, has a solution over $\Z$.    Lemma~\ref{linear Diophantine equation}   tells us that \eqref{the pivotal Diophantine eqn}  has a solution with
\begin{equation} \label{bounds on x and x hat}
	  |x|,|\hat{x}|   \ \leq \   \max\set{|u|, |\hat{u}|, |p-q|}.    
\end{equation}  
If $U$ and $V$ are expressed as words on $a$, $b$ and $c$ whose lengths sum to $n$, then $|u| + |\hat{u}| \leq n$ and $|p| + |q| \leq n^2$.   
We may take the entry $r$ in $W$ to be $0$, and then given  \eqref{bounds on x and x hat}, we can represent $W$ as a word  on $a$, $b$ and $c$ of length at most $2n^2$.

Next we will show that $\CL_H(n) \succeq n^2$. 
Suppose $n \in \N$.  Let $u = b [a^n, b^n]$ and $v=b$.  Then $|u| = 4n +1$ and $|v| =1$. These words represent conjugate elements of $H$ because  $$u \ = \  b [a^n, b^n]  \ = \  bc^{n^2} \ = \ a^{-n^2} b a^{n^2}$$ in $H$. For  $h = a^{\alpha} b^{\beta} c^{\gamma} \in H$ the condition for $uh=hv$ in $H$ is $$bc^{n^2} \, \left( a^{\alpha} b^{\beta} c^{\gamma} \right)   \ = \       \left(a^{\alpha} b^{\beta} c^{\gamma}\right) \, b,$$ or equivalently, 
$$ a^{\alpha} b^{\beta  + 1} c^{n^2 +   \alpha + \gamma}   \ = \       a^{\alpha} b^{\beta + 1} c^{\gamma}.$$   Comparing the powers of $c$ we   find $\alpha = -n^2$.  If a word $w$ represents $h$ in $H$, then the exponent sum of the letters $a^{\pm 1}$ in $w$ is $\alpha$. So  $|w| \geq n^2$.   
\end{proof}

\begin{Remark}  \label{higher diml Heisenberg}
The  $(2m+1)$-dimensional integral Heisenberg group $\mathcal{H}_{2m+1}(\Z)$,  for $m = 1, 2, \ldots$ is the group  of integer matrices 
$$\begin{pmatrix}
1  & \beta_1 & \cdots & \beta_m & \gamma  \\
  & 1 &  &  & \alpha_1 \\
  & & \ddots & & \vdots \\ 
  & & & 1 & \alpha_m \\
  &  &  &  & 1 \\   
\end{pmatrix}$$
and is presented by  
 \begin{equation*} \label{pres higher Heisenberg}
  \left\langle \; a_1, \ldots, a_m, b_1, \ldots, b_m ,c \; \left| \; [a_i,b_i]c^{-1}, \; [a_i,a_j],  \;  [b_i,b_j],  \; [a_i,b_j], \; [a_i,c], \; [b_i,c] \ \   \forall i\neq j  \; \right.  \right\rangle.
\end{equation*}
Following the proof of Theorem \ref{Heisenberg group thm}, one sees
that  $\mathcal{H}_{2m+1}(\Z)$ has $\CL(n) \simeq n^2$.  For the lower bound consider $b_1 c^{n^2} \sim b_1$.  For the upper bound the pivotal Diophantine equation \eqref{the pivotal Diophantine eqn} becomes 
\begin{equation} \label{longer pivotal Diophantine eqn}
	 \hat{u}_1 x_1 + \cdots \hat{u}_m x_m  - u_1 \hat{x}_1 - \cdots - u_m \hat{x}_m  \ = \ p  - q,   
\end{equation} 
and, assuming a solution exists, \cite{BFRT} tells us there is one with 
 	\begin{equation*} \label{bounds on x and xhat}
	  |x_i|,   |\hat{x}_i|    \ \leq \   \max\set{|u_1|, \ldots, |u_m|, |\hat{u}_1|, \ldots, |\hat{u}_m|, |p-q|}    
\end{equation*}  for all $i$.  
\end{Remark}

For class-$2$  finitely generated   nilpotent groups    Ji, Ogle, \& Ramsey \cite{JOR}  outline why the conjugator length functions  should grow at most polynomially.  In  \cite{BrRi1}  Bridson \& Riley establish explicit  upper bounds on the polynomial degree, and give examples showing those bounds to be optimal in general.

For general class-$c$ finitely generated nilpotent groups Macdonald,   Myasnikov,   Nikolaev,  \& Vassileva \cite[Thm.\ 4.7]{MMNV} obtain a polynomial upper bound of degree $2^m(6mc^2)^{m^2}$  on conjugator length functions,  where  $m$ is the number of elements in a Mal'cev basis for the group.

In Section~\ref{sec: Polynomial conjugacy length functions} we will describe families of examples from \cite{BrRi2, BrRi1} of nilpotent groups which exhibit polynomially growing conjugator length functions of all integer degrees. 

A complete understanding of the conjugator length functions of finitely generated nilpotent groups has yet to be found.

\subsection{Metabelian, polycyclic, and solvable groups}  \label{BS proof section}

\begin{Open problem} Which  functions arise
as  conjugator length functions of  finitely generated (or finitely presented) metabelian groups? What about finitely generated polycyclic groups?  What about groups of the form $\Z^d \rtimes \Z$?        
\end{Open problem}

Noskov proved that finitely generated metabelian groups have decidable conjugacy problem \cite{Noskov}.  Sale \cite{Sale3} explored the behaviour of conjugator length functions for group extensions leading, in particular,  to linear bounds for semi-direct products $\Z^d \rtimes_M \Z$ when  $M$ is  diagonalisable over $\mathbb{R}$ with positive eigenvalues.
Bridson and Riley~\cite{BrRi2} recently determined the conjugator length functions of the model filiform groups, which also have the form $\Z^d \rtimes \Z$ and are discussed further in Section~\ref{sec: Polynomial conjugacy length functions}.
A complete analysis of the  conjugator length functions of groups of the form $\Z^d \rtimes \Z$ is not yet available.  

Here, we present an elementary proof for an important special case of Sale's result: the metabelian Baumslag--Solitar groups  
$$\textup{BS}(1, m) = \langle a,s \mid s^{-1} a s = a^{m} \rangle.$$

\begin{thm}[\cite{Sale3}] \label{BS CL}
For all $m \geq 2$,  the conjugator length function of  $G= \textup{BS}(1, m)$   satisfies  $\CL(n) \simeq n$.	 
\end{thm}

\begin{proof}
The relations $a^{\pm 1}s =sa^{\pm m}$ and  $s^{-1} a^{\pm 1} = a^{\pm m}s^{-1}$   can  be used to shuffle each letter $s$ in a word on $\set{a, s}^{\pm 1}$ to the front and each $s^{-1}$ to the end without changing the element of $G$ it represents. 
Thus we can rewrite a word $w$  in the form  $s^p a^q  s^{-r}$ for some $p,q,r \in \Z$ with $p, r \geq 0$,  $p+ r \leq |w|$ and $\abs{q} \leq m^{|w|}$.   
Similar calculations show that for all $k  \in \Z$,  the distance $d_G(1,a^k)$ is at most a constant times $\ln( \abs{k} +1)$. 

The lower bound $\CL(n) \succeq n$ is witnessed by the conjugate elements  $s^n a s^{-n}$ and $a$.
	The obvious conjugator is $s^n$, for which we note $d_G(1,s^n) =n$.
	The centraliser of $a$ is  $\set{ s^p a^q s^{-p} \mid p,q\in \Z,  \ p\geq 0}$.
	Thus any conjugator between $s^n a s^{-n}$ and $a$ is of the form $s^{n+p} a^q s^{-p}$, and 
	$d_G(1, s^{n+p} a^q s^{-p})$
	 is minimised when $p=q=0$.

As for proving $\CL(n) \preceq n$, suppose $u$ and $v$ are words on $\set{a, s}^{\pm 1}$ such that $u \sim v$ in $G$.  We aim to show that   $\CL(u,v)$ is at most a constant times $n := |u| + |v|$.    
 
Rewriting $u$ and $v$ in the form $s^pa^qs^{-r}$ as above, and then replacing each with a suitable cyclic permutation of itself, we can assume $u = s^{\alpha} a^{\beta}$ and  $v = s^{\gamma} a^{\delta}$ in $G$ for some $\alpha, \beta, \gamma, \delta \in \Z$ with $\abs{\alpha} + \abs{\gamma} +  \ln( \abs{\beta} +1) +  \ln( \abs{\delta} +1)$   at most a constant times $n$.
The fact that $u$ and $v$ are conjugate implies that $\alpha = \gamma$ and, by replacing $u$ and $v$ by their inverses if necessary, we may assume $\alpha \geq 0$.

Suppose  we have an element of $G$ expressed as $w = s^p a^q  s^{-r}$ where $p,q,r \in \Z$ and $p,r \geq 0$.   
As $\alpha, p, r \geq 0$, 
\begin{eqnarray*}
uw & = \  s^{\alpha} a^{\beta} s^p a^q  s^{-r}  & = \ s^{\alpha + p} a^{\beta m^p +q } s^{-r}  \\
wv &  =  \   s^p a^q  s^{-r} s^{\alpha} a^{\delta}    & =  \ s^{\alpha + p} a^{q m^{\alpha} +\delta m^r} s^{-r} 
\end{eqnarray*}
in $G$.
So, as $a$ has infinite order in $G$, the condition  $uw=wv$ in $G$ is equivalent to 
\begin{equation}
q(m^\alpha -1)  \ = \  \beta m^p - \delta m^r. \label{characterization}
\end{equation}
In the case $\alpha =0$, this   reduces to $\beta m^p = \delta m^r$ and we see that there exists $k \in \Z$ such that $u s^k = s^k v$ in $G$ and $\abs{k}$ is at most a constant times  $\ln (\abs{\beta}+1) + \ln (\abs{\delta}+1) \leq n$, as required.       

By interchanging the roles of $u$ and $v$ and swapping $w$ for $w^{-1}$ if necessary, we may assume $p \geq r$. 
Then  \eqref{characterization}  implies that $m^r$ divides $q$ and so   $s^{p-r} a^{q/{m^r}}$ conjugates $u$ to $v$, so we may assume that $r=0$.  
 Since  for all $k \in \Z$ we have $u (u^k w)=(u^k w)v$ in $G$, 
we may replace $w$ with $w' = u^kw$, choosing $k$ so that $p' := p+ k \alpha$ satisfies $\alpha > p'  \geq 0$.    
Then $w' = (s^{\alpha} a^\beta)^{k} s^p a^q  = s^{p'} a^{q'}$ in $G$, 
  where $q' = (  \beta m^{p'} - \delta ) / (m^\alpha -1)$ by \eqref{characterization}.  
But then $$\abs{w'} \ \leq \ p' + \ln\left(\abs{q'} +1\right) \ < \ \alpha +    \ln\left( \beta m^{p'}  +  \abs{\delta}  +1\right),$$ which is at most a constant times $n$, as required.
\end{proof}

Sale also proved  that the conjugator length functions of the Lamplighter groups $(\Z / q\Z) \wr \Z$ and  of $\Z \wr \Z$ are linear \cite{Sale6}, and he obtained exponential upper bounds for some $\Z^d \rtimes \Z^k$ (with $k>1$) \cite{Sale3}, leading to  
exponential upper bounds on conjugator lengths of certain  conjugate pairs of elements of Hilbert modular groups and linear upper bounds for others.

As for  finitely generated polycyclic groups,  Remeslennikov~\cite{Remeslennikov}  proved them to be  conjugacy separable.  Since they are finitely presentable and have decidable word problem,  it follows that they have decidable conjugacy problems  and therefore recursive conjugator length functions (see Section~\ref{the CP and algorithms}).  We are not aware of further general bounds on the conjugator length functions of polycyclic groups.

There is little scope for general results about the conjugator length functions of finitely generated solvable groups because Kharlampovich~\cite{Kharlampovich} exhibited a finitely presented solvable group of derived length three with undecidable word problem and so undecidable conjugacy problem, meaning that no recursive function can bound  its conjugator length function from above. 

In \cite{Sale5} Sale showed that the conjugator length functions of free solvable groups are at most cubic.  To prove this,
he used the \emph{Magnus embedding}, which exhibits  free solvable groups as subgroups of wreath products,  thereby giving access to a more accessible setting in which to understand conjugacy. 
With this motivation,  Sale was also led to establish bounds on  conjugator length in wreath produces $A \wr B$, which he did in terms of the distortion functions of cyclic subgroups of $B$.

\subsection{Bestvina--Brady groups}  
 \label{Stallings section}

  Define  $$D \ = \ F(\alpha,\beta) \times F(\gamma,\delta) \times F(\varepsilon, \zeta).$$  Stallings' group $S$ is the kernel of the map $D  \to   \Z$ in which   $\alpha$, $\beta$, $\gamma$, $\delta$, $\varepsilon$, and $\zeta$ are all mapped to 
  the same generator of $\Z$.   We will use the finite presentation      
 \begin{equation} \label{pres}
S \ =   \  H \dot{\ast}_{K} \ = \ \left\langle \; a,b,c,d,t \; \left| \; [a,c], \; [a,d], \; [b,c], \; [b,d], \; t^a=t^b=t^c=t^d  \; \right.  \right\rangle
\end{equation}
of $G$  from \cite{BBMS, Gersten5}, which describes it as the HNN-extension
of $H = F(a,b) \times F(c,d)$ where the stable letter $t$ commutes with  the subgroup $K$ consisting of elements that can be represented by words on $\set{ a, b, c, d}^{\pm 1}$ of zero exponent-sum.  Stallings' constructed $S$ in \cite{Stallings} as the
first example  of a finitely presented group whose 3-dimensional integral homology
is not finitely generated --- in particular,  $G$ is not ``of type ${\rm{FP}}_3$.''  This  exoticism makes it an interesting case study for geometric invariants, including the conjugator length function.

\begin{thm} \label{Stallings CL}
The conjugator length function of  Stallings' group grows quadratically.	 
\end{thm}

\begin{proof} We start by proving that $\CL_G(n) \preceq n^2$.  
   Suppose $u$ and $v$ are words on $\set{a, b, c, d, t}^{\pm 1}$ representing non-trivial elements of$S$, that $u \sim v$,
    and $|u| + |v| \leq n$. 
    
 In the light of Corollary~\ref{cor:HNN simplify}  and the fact that  $\CL_H(n) \simeq n$ (because $\CL_{A \times B} \simeq \max\{\CL_A,\CL_B\}$ for all groups), it suffices to consider only the case where $u$ and $v$  admit an annular diagram $\Delta$ over the presentation \eqref{pres}  in which there is at least one $t$-corridor, and  all the $t$-corridors are radial, pairing off the letters $t^{\pm 1}$ in $u$ with  the $t^{\pm 1}$ in $v$.

\begin{figure}[ht]
\begin{overpic}[scale=1.0,unit=1mm]{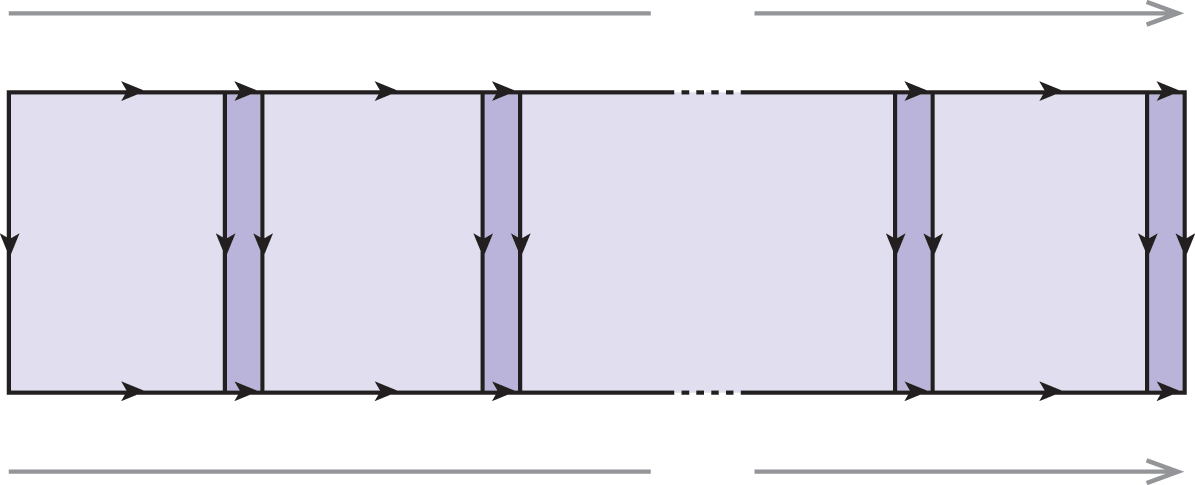}
 \put(59,39){\small{$u$}}     
 \put(59,0.5){\small{$v$}}     
 \put(10,35){\small{$u_1$}}     
 \put(32,35){\small{$u_2$}}     
 \put(87,35){\small{$u_l$}}     
 \put(10,4){\small{$v_1$}}     
 \put(32,4){\small{$v_2$}}     
 \put(87,4){\small{$v_l$}}
 \put(15,20){\small{$w_1$}}     
 \put(36,20){\small{$w_2$}}     
 \put(68,20){\small{$w_{l-1}$}}     
 \put(-3.5,20){\small{$w_l$}}
 \put(92,20){\small{$w_l$}}
 \put(102,20){\small{$w_l$}}
 \put(20,4){\small{$t^{\varepsilon_1}$}}     
 \put(42,4){\small{$t^{\varepsilon_2}$}}     
 \put(77,4){\small{$t^{\varepsilon_{l-1}}$}}
 \put(98,4){\small{$t^{\varepsilon_l}$}}
 \put(20,35){\small{$t^{\varepsilon_1}$}}     
 \put(42,35){\small{$t^{\varepsilon_2}$}}     
 \put(77,35){\small{$t^{\varepsilon_{l-1}}$}}
 \put(98,35){\small{$t^{\varepsilon_l}$}}
\end{overpic}
 \caption{A van~Kampen diagram for $u w = wv$ in $G$.}
  \label{radial t-corridors diagram}
\end{figure}
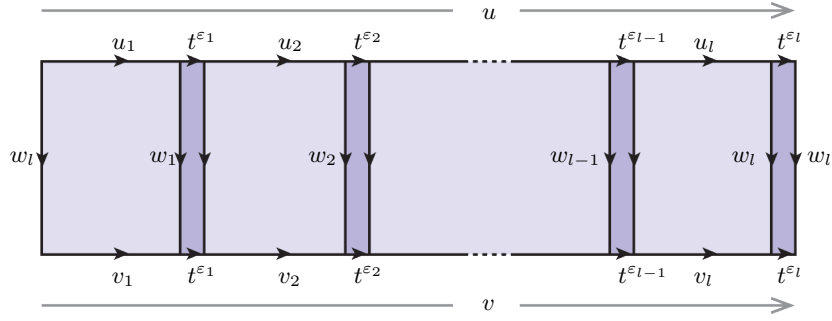

Replacing $u$ and $v$ by suitable cyclic permutations, we   have that  $u = u_1 t^{\varepsilon_1}  \cdots u_l t^{\varepsilon_l}$ and  $v = v_1 t^{\varepsilon_1}  \cdots v_l t^{\varepsilon_l}$, for some $\varepsilon_1, \ldots, \varepsilon_l = \pm 1$, where   $\bar{u} := u_1 \cdots u_l$ and $\bar{v} := v_1 \cdots v_l$  are the words obtained from $u$ and $v$ by deleting all $t^{\pm 1}$, and   $l \geq 1$ is the number of $t$-corridors in $\Delta$.  
Moreover, there exist   words $w_1, \ldots, w_l$ on $\set{a, b, c, d}^{\pm 1}$ representing elements of $H$
 (words labelling the sides of  the $t$-corridors in $\Delta$) such that $u_{i+1}   w_{i+1} = w_i  v_{i+1}$ in $H$ for all $i$ (indices mod $l$).    
 This is illustrated in Figure~\ref{radial t-corridors diagram} (which portrays the van Kampen diagram obtained by cutting $\Delta$ open along the side of the last $t$-corridor).
 In fact, by looking at the relators in $S$ involving $t$, we see $w_1,\ldots ,w_l \in K$. We also have  $u  w_l = w_l v$ in $G$ and $\bar{u} w_l = w_l \bar{v}$ in $H= F(a,b) \times F(c,d)$.

Let $z: H \to \Z$ be the map sending $a, b, c, d$ to a fixed generator of $\Z$.   So $K = \Ker \, z$.

Let $\theta_1\in F(a,b)$, $\phi_1\in F(c,d)$ be such that $\theta_1\phi_1$ is a minimal-length conjugator for $\bar{u}\sim\bar{v}$ in $H$.  We have $|\theta_1|+|\phi_1| \leq n$.
Let $\theta_2\in F(a,b)$, $\phi_2\in F(c,d)$ be such that $\bar{v} = \theta_2\phi_2$.
We also have $|\theta_2|+|\phi_2| \leq n$.
The set of \emph{all} $w \in H$ such that $\bar{u}w=w\bar{v}$ in $H$ is
$$W = \set{ \left. \theta_1  \theta_ 2^p  \phi_1  \phi_ 2^q  \,  \right| \, p, q \in \Z}.$$
We know there is some such $\theta_1  \theta_ 2^p  \phi_1  \phi_ 2^q$ in $K$, namely $w_l$.  So the equation $$p \, z(\theta_ 2)  +   q \, z(\phi_ 2)  \ = \   - z(\theta_1)  -   z(\phi_1)$$ has a solution $p, q \in \Z$.  So, by Lemma~\ref{linear Diophantine equation},  it has a solution with 
$\abs{p}, \abs{q} \leq \max \set{ \abs{z(\theta_ 2)}, \abs{z(\phi_ 2)}, \abs{ - z(\theta_1)  -   z(\phi_1)}  }$, and thus with $\abs{p}, \abs{q} \leq n$.   
Define $w' := \theta_1  \theta_ 2^p  \phi_1  \phi_ 2^q$ for these $p$ and $q$,  and note that $w'$ has length at most a constant times $n^2$.    

Suppose $i \in \set{1, \ldots, l}$.  
Since $w \in K$ and $u_i^{-1} \cdots u_1^{-1} w v_1 \cdots v_i  = w_i \in K$, we deduce that     $z(u_i^{-1} \cdots u_1^{-1}   v_1 \cdots v_i) =0$, and therefore, as $w' \in K$, we learn that   $w'_i :=  u_i^{-1} \cdots u_1^{-1} w' v_1 \cdots v_i    \in K$ also.  
Since $t$ commutes with $K$, we get $uw'=w'v$ in $S$, via a diagram like that in Figure~\ref{radial t-corridors diagram}, but with $w_i$ replaced by $w'_i$ and $t$-corridors replaced by van~Kampen diagrams for $t^{\varepsilon_i}w'_i=w'_it^{\varepsilon_i}$, for all $i$. 
 
\medskip

For the bound $\CL_S(n) \succeq n^2$,  take an arbitrary integer $n$ and consider
\begin{align*}
u & = \ \left(\alpha^n \beta^{-n}, (\gamma \delta^{-1})^n \gamma, (\varepsilon \zeta^{-1})^n \varepsilon^{-1} \right) \\
v & = \  \left(\beta^{-n} \alpha^n, (\gamma \delta^{-1})^n \gamma, (\varepsilon \zeta^{-1})^n \varepsilon^{-1} \right) \\
w & = \  \left(\alpha^n, 1, 1\right).
\end{align*}
Then $u$ and $v$ are elements of $S$ and $uw=wv$ in $D$.  
Let $C$ be the centraliser of $u$ in $D$. The set of all conjugators for $u\sim v$ in $D$ is $$Cw  \ = \  \set{ \left. \, g=  \left(  (\alpha^n \beta^{-n})^r \alpha^n, ((\gamma \delta^{-1})^n \gamma)^s, ((\varepsilon \zeta^{-1})^n \varepsilon^{-1})^t \right)  \,  \right| \, r,s,t\in \Z \, },$$ which intersects $S$ in $\set{ g \mid  n + s -t  = 0 }$. So $u\sim v$ in $S$ and any conjugator $g \in S$ has either $\abs{s} \geq n/2$ or $\abs{t}  \geq n/2$.  Either way, we get that $\abs{g}_D \geq n^2$.  It follows that     $\abs{g}_S \succeq n^2$ (with respect to any fixed finite generating set for $S$) since $\abs{h}_S   \succeq \abs{h}_D$ for all $h \in S$.  
 \end{proof}
 
 Stallings' group is an example of a  \emph{Bestvina--Brady group} --- the kernel of the map from a right-angled Artin group $A$ to $\Z=\<h\>$ that is defined by sending all the standard generators to $h$.  Both these and more general kernels are known to have decidable conjugacy problem thanks to   \cite[Theorem~3.1]{Bridson12}.  
    
    \begin{Open problem}
    Estimate the conjugator length functions of Bestvina--Brady groups.
    \end{Open problem}

\subsection{Mapping class groups, 3-manifold groups, right-angled Artin groups (RAAGs),  graph products,  
Coxeter groups,  virtually special groups, and  hierarchically hyperbolic groups}   Hemion  gave the first  solution to the conjugacy problem for mapping class groups \cite{Hemion},  Pr\'{e}aux for 3-manifold groups \cite{Preaux1, Preaux2}, and Servatius for RAAGs \cite{Servatius}. Some of the other groups in the list above are covered by the solution to the conjugacy problem for semihyperbolic groups \cite{AlonsoBridson} discussed above, including fundamental groups of special cube complexes, Coxeter groups, and certain graph products.

Mapping class groups have linear conjugator length function:  Masur \& Minsky \cite{MM00} obtained a linear upper bound 
on $\CL(u,v)$ for   pairs of conjugate pseudo-Anosov elements and the work of  J.\ Tao \cite{Tao} covers the remaining cases.  Behrstock \& Dru\c{t}u \cite{Behrstock-Drutu} gave new proofs of these bounds  and established quadratic upper bounds on the conjugator length functions of fundamental groups of prime 3-manifolds --- some classes of which are also covered in \cite{JOR, Sale3}. From Servatius' solution to the conjugacy problem for RAAGs \cite{Servatius} it follows that the conjugator length function grows linearly. (First cyclically reduce both input words.  Then a sequence of cyclic permutations and applications of commutator relations will take one cyclically reduced word to the other, assuming the input words were conjugate.)  
More generally, Genevois recently gave bounds on the conjugator length functions of graph products \cite{Genevois}.

Generalising \cite{MM00,Behrstock-Drutu}, Abbott \& Behrstock \cite{AbBe} established upper bounds on the conjugator lengths of 
certain pairs of elements in hierarchically hyperbolic groups.   In the special case of conjugate pairs 
of infinite order elements of virtually compact special   groups, they obtain a linear upper bound.

\subsection{Lattices in Lie groups}  Grunewald \& Segal \cite{GrunewaldSegal} solved the conjugacy problem for arithmetic groups; according to the Margulis Arithmeticity Theorem, this covers all irreducible lattices in higher rank semismiple Lie groups.  The behaviour of the conjugator length function for such lattices is almost wholly unknown, even  for $\SL_n(\Z)$,
but Sale~\cite{Sale4} established linear bounds on $\CL(u,v)$ for $u\sim v$ hyperbolic or unipotent elements of any semisimple Lie group.

\subsection{Free-by-cyclic groups} For every group $G= F \rtimes \Z$, where $F$ is a finite-rank free group,
there is an  algorithm solving the conjugacy problem   \cite{BMMV, BridsonGroves}. 

\begin{Open problem} \label{Op:FZ}
	Are the conjugator length functions of free-by-cyclic groups all linear?
\end{Open problem}

When $G$ is hyperbolic its conjugator length function is linear (see Section~\ref{sec:hyp groups}).  
To approach the general case, one can use the fact that every free-by-cyclic group is hyperbolic relative to a  family of free-by-cyclic subgroups,  where  the defining automorphisms of these subgroups are polynomial-growing --- see \cite{BFWOld} for history and references.  Antol\'in and Sale proved \cite{AS} that if a group $G$ is non-degenerately hyperbolic relative to the family of  subgroups $H_\omega\ (\omega \in \Omega)$, then $\CL_G(n) \simeq \max\{\CL_{H_\omega}(n)  : \omega\in\Omega\}+n$.  This reduces  Problem \ref{Op:FZ} to the case where the defining automorphism
is polynomially growing.  In \cite{BrRiSa2} we prove that a prototype family of such groups have linear conjugator length functions, namely the groups $F \rtimes_\varphi \Z$ where 
  $F = F(a_1, \ldots, a_m)$ and 	$\varphi(a_i) = a_ia_{i-1}$  for $2\leq i \leq m$  and $\varphi(a_1)=a_1$.

\subsection{One-relator groups}  Famously, the conjugacy problem for one-relator groups remains open in general, so the question of bounding their conjugator length functions is out of reach.  Gillis \cite{Gillis2} recently proved that the conjugator length function of the Baumslag--Gersten group $\langle a,b \mid (b^{-1}a b)^{-1} a (b^{-1}a b) = a^2 \rangle$ is bounded above and below by towers of exponentials of logarithmic height and  asked whether this group has the fastest growing conjugator length function among all one-relator groups, echoing a longstanding question about the Dehn functions of these
groups.

\subsection{Thompson's group $F$ and its cousins}   \label{Thompson's group etc section} The conjugacy problem for Thompson's group  $V$ was solved by Higman  \cite[Theorem~9.8]{Higman2} and for $F$ it was solved by Guba \& Sapir  \cite{GubaSapir2}.    
Belk \&   Matucci \cite{BM1} prove that the conjugator length of $F$ is quadratic, and establish a quadratic lower bound for Thompson's groups  $T$ and $V$.  They conjecture upper bounds of $n^2$ and $(n \log n)^2$, respectively, for $T$ and $V$.      

\subsection{Permanence results}  \label{permanence} It is  easy to see that the conjugator length function of a direct product is the maximum of the conjugator length functions of the factors.
 The same upper bound holds for free products --- a special case of the result for relatively hyperbolic groups \cite{AS}.  
 
 There are upper bounds on the conjugator length functions of wreath products  in \cite{Sale5}.
 For more general group extensions, conjugacy can be complicated, as evidenced by  the examples of Collins \& Miller where the solvability of the conjugacy problem does not pass to or from index-2 subgroups \cite{CollinsMiller}.  Methods for estimating upper bounds on conjugator length functions for some types of extension can be found in \cite{Sale3}.  

 The infinite dihedral group $D_{\infty} = \langle a,b \mid a^2, \, b^2 \rangle$ is an elementary example where passing to a finite index subgroup, namely $\Z$, disrupts conjugator length. The conjugator length function of $D_{\infty}$ is  linear but unbounded, whereas  $\CL_\Z(n) = 0$.

\subsection{Locally compact groups of Euclidean isometries} \label{Euclidean isoms} In a recent paper \cite{SantosRegoSchwer} Santos Rego and Schwer extend the definition of conjugator length to locally compact groups and show that the conjugator length functions of \emph{split} locally compact subgroups $H$ of the full isometry group $\textup{Isom}(\mathbb{E}^n)$ of Euclidean $n$-space grow at most linearly.  The group  $\textup{Isom}(\mathbb{E}^n)$ can be expressed as the semi-direct product  $\mathbb{R}^n \rtimes O(n)$ of translations and rotations.  To say $H \leq \textup{Isom}(\mathbb{E}^n)$ \emph{splits} means it inherits a semi-direct product structure $H = T_H \rtimes H_0$ where  $T_H$ and $H_0$ are the intersections of $H$ with the $\mathbb{R}^n$ and $O(n)$ factors, respectively.  Examples of  $H$ where their theorem applies include affine Coxeter groups and and split crystallographic groups.

\subsection{Groups with quadratic Dehn functions}  Hyperbolic groups are precisely the finitely presentable groups with the slowest possible growing Dehn functions, namely linear or, equivalently,  strictly subquadratic. A wide assortment of finitely presented groups have the next slowest growing Dehn functions --- namely, quadratically growing.  Many of these have conjugator length functions growing at most linearly --- for example,  the free-by-cyclic groups of \cite{BrRiSa2}.  Stallings' group   and Thompson's group $F$  both have quadratic Dehn functions (see \cite{DERY},    \cite{Allcock, OS4}, and \cite{Guba}), but have conjugator length functions  that grow at least quadratically (see Sections~\ref{Stallings section} and \ref{Thompson's group etc section}). Olshanskii \& Sapir \cite{OS6}, solving a problem of Rips,  constructed the first example of a group with quadratic Dehn function and undecidable conjugacy problem. By virtue of our discussion in Section~\ref{the CP and algorithms}, its conjugator length function must be non-recursive.

\subsection{List conjugacy} Bridson \& Howie \cite{BH} generalized  Theorem~\ref{thm: CL for hyperbolic groups}\ref{hyp linear bound} to lists of elements in a hyperbolic group.  They prove that a group $G$ with finite generating set $A$ is  $\delta$-hyperbolic, then there exists $C>0$ such that for all $m \in \N$ and all words $u_1, \ldots, u_m, v_1, \ldots, v_m$ on $A$ such that there exists a  word  $w$ on $A$ for which  $u_i w = w v_i$ in $G$ for all $i$, there exists such a $w$ of length at most $C \max( u_i, v_i) + C$. This bound, and associated results they give on the time complexity of the list-conjugacy problem, are significant for cryptography (cf.\ Section~\ref{Cryptography}).  They  indicate that hyperbolic groups are not a suitable platform for the Anshel-Anshel-Goldfeld protocol. 

In an appendix to  \cite{BH}, Bridson \& Howie exhibit a finitely presented group,  constructed using ideas from \cite{CollinsMiller}, for which the conjugacy problem is decidable but the list-conjugacy problem is not.   It follows that conjugator length for pairs of elements in a group and conjugator length for pairs of lists from the same group can be starkly different.   
 
Beyond these results,  the study of conjugator length for lists has scarcely been touched.

\section{What functions are conjugator length functions?} \label{sec: what functions?}

In Section~\ref{survey} we encountered groups for which the conjugator length function is
zero (e.g., $\Z$), linear (e.g., the rank-2 free group), quadratic (e.g., the 3-dimensional
integral Heisenberg group and Stallings’ group), and not bounded from above by
any recessive function $\N \to \N$ (e.g., examples of Collins \& Miller and of Olshanskii \&
Sapir). This leads to the question: what functions can be conjugator length functions?

This survey focuses on finitely presented groups, but the question is also natural for finitely generated groups.  There, Goffer,  Mihaila, and Osin \cite{GMO} have recently shown that with the luxury of infinitely many defining relations to encode information, it is possible to realize any non-decreasing function that is either bounded or at least linear as a conjugator length function.  Their construction is a variant of an HNN-extension, and their proofs are via diagrams.  Moreover, strikingly, they show that any two non-decreasing functions can be realized as the conjugator length functions of two commensurable finitely generated groups. 

What functions, though, can be conjugator length functions of finitely presented groups?
There are different approaches that one might take to this question: for example, one might explore or construct explicit groups whose 
 conjugator length functions exhibit novel behaviour, or one might concentrate more on existence theorems, for example
 encoding the time functions of Turing machines, or other machines, into conjugator length functions. We shall concentrate
 on the former approach, which is where our   research contributions have been.  We will conclude with some comments on recent developments via the latter approach.

\subsection{Polynomial conjugacy length functions} \label{sec: Polynomial conjugacy length functions}

The first families of groups displaying polynomial conjugacy length functions of all integer degrees  appeared in a pair of companion papers by the first two authors.  The \emph{discrete model filiform groups} $\G_d = \Z^d\rtimes_\phi\Z$ of  \cite{BrRi2}, where the $\Z^d$-factor has            basis $a_1,\dots,a_d$ and the automorphism $\phi$ fixes $a_d$ and maps $a_i\mapsto a_i a_{i+1}$ for $i=1,\dots,d-1$, are shown to have conjugator length functions $\simeq n^d$.  The class-2 nilpotent groups  $G_m$ from \cite{BrRi1} are defined for all $m \geq 1$ as the central extensions of $\Z^{m+2} = \langle a_1, \ldots, a_m, b_1, b_2\rangle$ by $\Z^m  = \langle c_1, \ldots, c_m \rangle$ such that every pair among the generating set $a_1, \ldots, a_m, b_1, b_2, c_1, \ldots, c_m$ commutes except that for  $i=1, \ldots, m$, instead of $a_i$ commuting  with $b_1$, we have  $b_1 a_i = a_i b_1 c_i$, and for  $i=1, \ldots, m-1$, instead of $a_i$ commuting  with $b_2$, we have $b_2 a_i = a_i b_2 c_{i+1}^{-1}$. 
The group $G_m$ has conjugator length function $\simeq n^{m+1}$.

The groups  $\G_d$ are lattices in the model filiform Lie groups.  Their conjugator length functions are found by inducting on $d$, using that $\G_{d}$ is isomorphic to $\G_{d+1}$
modulo its  cyclic centre; the argument proceeds via a careful
analysis of the geometry of cyclic subgroups and centralisers.  

By contrast, the groups $G_m$ are bespoke.  As class 2-nilpotent groups they admit standard normal forms, and the conjugacy condition can be studied by comparing normal forms: for a pair of elements $g$ and $h$, being conjugate amounts to a certain system of linear Diophantine equations having a solution; and the set of all conjugators can be understood in terms of the set of all solutions to that system.  Our proof here of Theorem~\ref{Heisenberg group thm} serves as a prototype for this type of argument.  The groups $G_m$ are constructed so that this system of Diophantine equations has a recursive structure that allows one to determine $\CL_{G_m}(n)$.

The  groups  $G_m$ exhibit a  sharp difference in behaviour between
the Dehn function and the conjugator length function; the former grows $\preceq n^3$ for all $m$ \cite{GHR, Gromov6}, but the latter grows $\simeq n^{m+1}$.  (The Dehn function of $\Gamma_d$ grows $\simeq n^{d+1}$ for all $d$ \cite{BMS,  BridsonPittet}.)

\subsection{Snowflake groups and conjugacy length functions of non-integer degree }  \label{s:snowflakes}

In  the recent paper \cite{BrRi4} the first two authors constructed finitely presented groups with conjugator length functions that grow like $n^{e}$ for a  set of exponents $e$ that is dense in $[2, \infty)$.  The examples begin with   the   \emph{snowflake groups}
 \begin{equation*} 
\left\langle a, b, s, t \mid [a,b]=1,\, s^{-1}a^qs=a^pb,\ t^{-1}a^qt=a^pb^{-1} \right\rangle
 \end{equation*}
 of Bridson \& Brady \cite{BB} which are defined for integers $p>q>0$ and have Dehn functions growing $\simeq n^{2\alpha}$, where $\alpha=\log_2(2p/q)$.  It is shown in \cite{BrRi4} that these groups have linear conjugator length. However,   one can form a certain 2-step HNN-extension
  \begin{equation*} 
 \left\langle a, b, s, t, \theta, z  \ \left| \ \parbox{46mm}{$z$ central, $[b,\theta]=z$,  $[a,b]=1$, \\ $s^{-1}a^qs =a^pb$, $t^{-1}a^qt=a^pb^{-1}$} \right. \right\rangle
 \end{equation*}  
  that  encodes the Dehn function of the snowflake groups into  the distortion of an infinite cyclic subgroup,  and one can then prove that this distortion is faithfully reflected in the conjugator length function of the resulting group --- at a basic level, this too draws on ideas that we saw in the example of the Heisenberg group (Theorem~\ref{Heisenberg group thm}), however the translation to conjugator length has to be mediated through  what we call $\zeta$-maps, which will be explained in Section~\ref{zeta maps}.  

Famously, the interval $(1,2)$ is a gap in the spectrum of Dehn function exponents.  We see no reason to expect the same for conjugator length functions.

    \begin{Open problem}
Is there a dense set of exponents $e$  in the interval $[1,2]$ for which there exist finitely presented groups with conjugator length functions growing like $n^{e}$?    
\end{Open problem}

\subsection{Promoting subgroup distortion to conjugator length}

Our strategy for obtaining a lower bound for the conjugator length function of the Heisenberg group (Theorem~\ref{Heisenberg group thm}) can be leveraged to the following method for constructing   groups with fast-growing conjugator length functions. There are many ways to distort $\Z$-subgroups in finitely presented groups.  Indeed, Olshanskii \cite{Ol3} showed that, modulo straight-forward necessary conditions, every computable function $\N \to \N$ grows like the distortion function of such a subgroup.   There are also elementary constructions of extravagant distortion: $\langle a \rangle$ is exponentially distorted  in $\textup{BS}(1, 2) = \langle a,s \mid s^{-1} a s = a^2 \rangle$; more generally, the distortion of $\langle s_0 \rangle$  in $\langle s_0,  s_1, \ldots, s_k \mid  s_{i}^{-1} s_{i-1} s_{i} = s_{i-1}^2 \textup{ for } i=1, \ldots, k \rangle$  grows like a $k$-fold iterated exponential function; and  $\langle a \rangle \leq \langle a, t \mid (t^{-1} a t )^{-1} a (t^{-1} a t ) = a^2 \rangle$ is distorted yet faster --- see e.g., the survey \cite{RileyNotes}.  So the theorem below leads to elementary examples of groups with fast-growing conjugator length functions.

\begin{thm} \label{Free-product with amalgam theorem}Let $H$ denote the 3-dimensional Heisenberg group with 
center $\<c\>$.
Suppose $\Lambda$ is a group with finite generating $S$ and with an infinite cyclic subgroup $Z = \langle \lambda \rangle$.
Let  $d  = \Dist_Z^{\Lambda} : \N \to \N$ be the \emph{distortion function}   for $Z$ in $\Lambda$, that is, 
$$d(n) \ := \  \max \set{  \  |r|  \  \mid  \     \abs{\lambda^r}_S  \leq n   \ }.$$	
Then, the conjugator length function of the free-product-with-amalgamation $$\Sigma = \Lambda \ast_{\lambda = c} H,$$ 
satisfies 
	$$\max \set{n^2, d(n)} \ \preceq \   \CL_{\Sigma}(n).$$
\end{thm}

\begin{proof}
Fix $n \in \N$.   Suppose $\sigma$ is a  minimal length  word on $S$ that represents ${\l}^{d(n)}$ in $\Lambda$.  So $|\sigma| \leq n$.  Let $u = b \sigma$.  Then $|u| = n +1$, and $u \sim b$ in   $\Sigma$, because in   $\Sigma$,   $$u      \ = \ b {\l}^{d(n)}  \ = \  bc^{d(n)} \ = \ a^{-d(n)} b a^{d(n)}.$$

Now suppose $r \in \Z$ and $w$ is a word on $S \cup \set{a,b,c}$ such that $bc^r w=w b$  in $\Sigma$.  So the word $\nu = bc^r w b^{-1} w^{-1}$  represents $1$  in $\Sigma$.  Express $w$ as $\sigma_1 \mu_1 \cdots \sigma_m \mu_m$ where $\sigma_1, \ldots,  \sigma_m$ are subwords on $S$ and  $\mu_1, \ldots,  \mu_m$ on   $\set{a,b,c}$, and all are non-empty with the possible exceptions of $\sigma_1$ and $\mu_m$.  Let $L(w)$ be the number of $\sigma_1, \ldots,  \sigma_m, \mu_1, \ldots,  \mu_m$ that are  non-empty.    

Assume that $w$ is not a word on $\set{a,b,c}$. As words we have  $$\nu  \ = \  bc^r  \, \sigma_1 \mu_1 \cdots \sigma_m  \mu_m \, b^{-1} \,  \mu_m^{-1} \sigma_m^{-1} \cdots \mu_1^{-1}  \sigma_1^{-1}$$ 
and, because  $\Sigma$ is a  free product with amalgamation and   $w$ is not a word on $\set{a,b,c}$,  
the normal form theorem for amalgamated free products tells us that either (i) some non-empty word $\sigma_i$ represents a    $\lambda^{\alpha}$ in $\Lambda$ for some $\alpha \in \Z$, or (ii) $\sigma_1$ is the empty word and $bc^r  \mu_1$ represents  $c^{\beta}$ in $H$ for some $\beta \in \Z$, or (iii) $\mu_m \, b^{-1} \,  \mu_m^{-1}$ represents a  power of $c$ in $H$. In the event of (i), replace this subword $\sigma_i$ in $\nu$ by  $c^{\alpha}$ and also replace the $\sigma_i^{-1}$ by $c^{-\alpha}$.  In the event of (ii) $bc^r  \mu_1 = c^{\beta}$ and so $\mu_1^{-1} =  bc^{r - \beta}$ in $H$, and we replace the subword $bc^r  \mu_1$ in $\nu$ by $c^{\beta}$ and the subword $\mu_1^{-1}$ by $bc^{r - \beta}$ to give  $$c^{\beta} \sigma_2 \mu_2 \cdots \sigma_m  \mu_m \, b^{-1} \,  \mu_m^{-1} \sigma_m^{-1} \cdots \mu_2^{-1}  \sigma_2^{-1}bc^{r - \beta},$$ which is conjugate to  $$bc^{r} \sigma_2 \mu_2 \cdots \sigma_m  \mu_m \, b^{-1} \,  \mu_m^{-1} \sigma_m^{-1} \cdots \mu_2^{-1}  \sigma_2^{-1}.$$   Case (iii) cannot occur, because  the map $H \onto \Z^2$ that kills $c$ reveals that no such equality holds in $H$.   

Thereby, via Case (i) or (ii), we have  a word $\overline{w}$ such that  $bc^r \overline{w} b^{-1} \overline{w}^{-1}$  represents $1$  in $\Sigma$ and  $L(\overline{w}) < L(w)$.  Further, the substitutions and the conjugation we used to obtained $\overline{w}$ from $w$ do not change the exponent-sum of the $a$-letters in $w$.  

So, by induction, there is a word  $\hat{w}$ on $\set{a,b,c}$  such that $bc^r \hat{w} b^{-1} \hat{w}^{-1}$  represents $1$  in $\Sigma$ (and so in $H$) and the   exponent-sum of the $a$-letters in  $w$ and $\hat{w}$ are equal.  We saw in our proof of Theorem~\ref{Heisenberg group thm} that such word  $\hat{w}$ contains at least $|r|$ occurrences of the letter $a$.  So the same is true of $w$ and therefore  $|w| \geq |r|$.  

Applying this  to  $r=d(n)$ proves that $d(n) \ \preceq \   \CL_{\Sigma}(n)$.  And from Theorem~\ref{Heisenberg group thm} we have $n^2 \ \preceq \   \CL_{H}(n)$, and so  $n^2 \ \preceq \   \CL_{\Sigma}(n)$.  
\end{proof}

  What is missing from Theorem~\ref{Free-product with amalgam theorem} is a matching upper bound on $\CL_{\Sigma}$.  Our next theorem sets out another construction which promotes distortion functions to conjugator length functions.  This  achieves matching upper and lower bounds but, for reasons we will explain at the end of this section, it too is not fully satisfactory.

\begin{thm} \label{central distortion versus CL}
 Let $\Lambda$ a  group with finite presentation $\langle A \cup \set{ \lambda} \mid R \rangle$ where $A$ is a finite set and $R$ is a finite set of words on $A \cup \set{ \lambda}$.  Assume that the subgroup $\langle \lambda  \rangle \leq \Lambda$ is central and infinite cyclic, and that its distortion function satisfies    $\CL_{\Lambda}   \preceq \Dist_{\< \lambda\>}^{\Lambda}$.    
  Then   the conjugator length function of  
$$\Sigma \ := \ \left\langle \  A \cup \{\lambda, p, q, r \} \ \mid \  R \cup \{  \,  [p,r], \  [q, \lambda  r  ],  \ [\lambda, r ]  \, \} \ \right\rangle$$  
satisfies   $\CL_{\Sigma} \ \simeq \ \Dist_{\< \lambda\>}^{\Lambda}$.
\end{thm}

Before we prove Theorem~\ref{central distortion versus CL}, consider  $\Lambda' = \langle \  A \cup \set{ \lambda, r} \mid R \cup \set{ [\lambda , r]} \ \rangle$, where $r$ is a symbol  not in $A \cup \set{ \lambda}$.    

\begin{lemma} \label{CL reduction}    $\CL_{\Lambda'}\mleft(n\mright)\leq  \CL_{\Lambda}\mleft(n\mright) +  2n$ for all $n \in \N$.  
\end{lemma}

\begin{proof}
In this scenario, $\Lambda' = \Lambda \dot{\ast}_{\langle \lambda \rangle} $ is an HNN-extension of $\Lambda$ with stable letter $r$, so we can apply  Corollary~\ref{cor:HNN simplify} to see that,  at a cost of adding at most $n$ to conjugator length, it suffices to analyze the conjugator lengths of pairs of words $u$ and  $v$  on  $(A \cup \set{\lambda,r})^{\pm 1}$ that represent non-identity conjugate  elements of $\Lambda'$ and admit     an annular diagram  $\Omega$ in which (1) there is at least one $r$-corridor, and (2) all the $r$-corridors are \emph{radial} (meaning that it is a chain of faces connecting an edge labeled by $r^{\pm 1}$ in the  boundary component labeled by $u$ to an  $r^{\pm 1}$ in that labeled by $v$).

Cutting such an $\Omega$ along one side of  a radial $r$-corridor gives a van~Kampen diagram for $\lambda^m    \hat{u} \lambda^{-m} \hat{v}^{-1}$  for some  cyclic permutations   $\hat{u}$ and $\hat{v}$  of $u$ and $v$ (respectively) and some $m \in \Z$.  
But $\lambda$ is central in $\Lambda'$, so $\hat {u} = \hat{v}$ in $\Lambda'$, and so $\CL_{\Lambda'}(u,v) \leq n$.  
\end{proof}

\begin{proof}[Proof of Theorem~\ref{central distortion versus CL}.]
First we will show  that  $\Dist_{\< \lambda\>}^{\Lambda} \preceq \CL_{\Sigma}$. 
Suppose $w$ is a  word on $\left(A  \cup \set{ \lambda}\right)^{\pm 1}$ equal to $\lambda^{\alpha}$ in $\Lambda$ for some $\alpha \in \Z$.     Then $pq^{-1} \sim pwq^{-1}w^{-1}$  in $\Sigma$ via an annular diagram shown schematically in Figure~\ref{binary tree}.      
We claim that this and indeed \emph{any}  annular diagram $\Omega$ demonstrating $pq^{-1} \sim pwq^{-1}w^{-1}$ in $\Sigma$ contains $r$-annuli that are essential  (meaning they circle the hole in $\Omega$) and are   nested to a depth of $\abs{\alpha}$.  Then $\Dist_{\< \lambda\>}^{\Lambda}\mleft(n\mright)     \leq    \CL_{\Sigma}\mleft(n\mright)$ will follow because any path through $\Omega$ connecting $\ast_u$ to $\ast_v$ must cross each of those annuli.

There is one radial $p$-corridor  and one radial  $q$-corridor in $\Omega$.  The word along the sides of this $p$-corridor may or may not be freely reduced, but  must freely equal $r^{\beta}$ for some $\beta \in \Z$. Likewise, the word along the sides of the $q$-corridor must freely equal $(\lambda r)^{\gamma}$ for some $\gamma \in \Z$.  
(It may be that the $p$- or $q$-corridor has  zero area, in which case $\beta =0$ or $\gamma=0$.)
But then  $r^{-\beta}(\lambda r)^{\gamma}  = w$ in $\Sigma$, so $(\lambda r)^{\gamma} =    r^{\beta}  \lambda^{\alpha}$. (See Figure \ref{binary tree}.)
By killing $p$, $q$,  and $r$, we see that  $(\lambda r)^{\gamma} =    r^{\beta}  \lambda^{\alpha}$  in $\Sigma$  induces an equality  in $\langle \lambda \rangle = \Z$ which tells us that   $\gamma = \alpha$.    
We can then map onto $\langle r\rangle =\Z$ by killing all other generators, giving $\beta = \gamma$.
No $r$-corridors terminate on the boundary of $\Omega$, so all form $r$-annuli.
If   an $r$-corridor emanating from the side of the radial $p$-corridor intersects the radial $p$-corridor twice, then there would be a sub-disc-diagram whose boundary word is $x y^{-1}$ for some word $x$ on $r^{\pm 1}$ read along a portion of the side of that radial $p$-corridor, and a word $y$ along a portion of the side of that $r$-corridor (as in the left annular diagram of Figure~\ref{fig:radial-spiral-intersection}).  But $y$ contains no  $r$ and so $x$ freely equals the empty word on account of the retraction $\Sigma \onto \langle r \rangle =\Z$.
 It follows that  there must be circling $r$-corridors  nested to a depth $\abs{\alpha}$.    
 
\begin{figure}[ht]
\begin{overpic}[scale=1.0,unit=1mm]{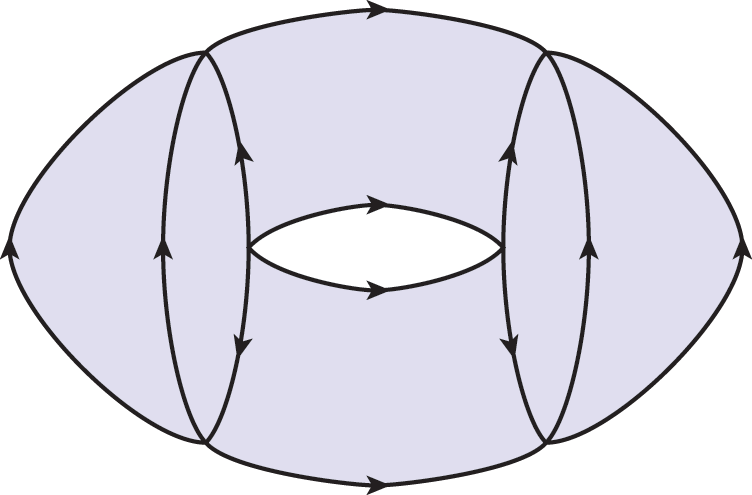}
 \put(31,-2){\small{$p$}}     
 \put(31,14){\small{$p$}}     
 \put(31,27){\small{$q$}}     
 \put(31,43){\small{$q$}}    
  \put(-3,20){\small{$w$}}     
   \put(64,20){\small{$w$}}     
   \put(9.5,20){\small{$\lambda^{\alpha}$}}
      \put(51,20){\small{$\lambda^{\alpha}$}}
     \put(21.3,29){\small{$(\lambda r)^{\gamma}$}}
          \put(35.4,29){\small{$(\lambda r)^{\gamma}$}}
     \put(21.5,11){\small{$r^{\beta}$}}
          \put(38.7,11){\small{$r^{\beta}$}}
\end{overpic}
 \caption{An annular van Kampen diagram over the group $\Sigma$.}
  \label{binary tree}
\end{figure}
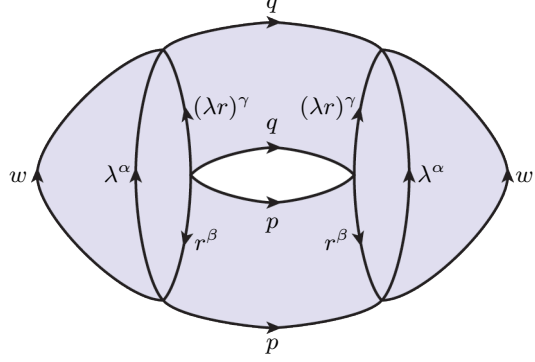

Now we prove the reverse bound, namely  $\CL_{\Sigma} \preceq \Dist_{\< \lambda\>}^{\Lambda}$. 
Suppose $u$ and $v$ is a pair of words  with  $u \sim v$ in $\Sigma$.  Let $n = |u| + |v|$.  
 
If $u$ and $v$ both contain no letters $p$ or $q$, then they are conjugate in $\Lambda'$ and $\CL_{\Sigma}\mleft(u,v\mright) \leq \CL_{\Lambda'}\mleft(u,v\mright)$, and Lemma~\ref{CL reduction} tells us that this is within the required bound given that $\CL_{\Lambda}   \preceq \Dist_{\< \lambda\>}^{\Lambda}$.

Suppose now that  $u$ and $v$ together contain at least one letter $p$ or $q$.  Let $\Omega$ be an annular diagram  of minimal area (that is, minimal number of faces) for $u$ and $v$. Now, $\Sigma$ is a multiple HNN-extension of $\Lambda'$  with  stable letters $p$ and $q$. So, given  Corollary~~\ref{cor:HNN simplify}, we may assume that every letter $p$ and $q$ in $u$ or $v$ is the start of a radial corridor in $\Omega$ and $\Omega$ contains no $p$- or $q$-annulus.
 We will prove that then 
\begin{equation}
\CL_{\Sigma}\mleft(u,v\mright) \ \leq \  \Dist_{\langle \lambda \rangle}^{\Lambda}\mleft(n\mright) +    \CL_{\Lambda'}\mleft(n\mright)  +  4n.  \label{bound we want}
\end{equation}
The result will then follow by Lemma~\ref{CL reduction}.  We will consider three cases.

\emph{Case 1: $\Omega$ contains both a   $p$- and a   $q$-corridor.}   As $p$- and $q$-corridors cannot cross (as no defining relation contains both $p$ and $q$ letters),   such a pair of radial corridors   together with portions of the inner and outer boundary of $\Omega$ bound a disc sub-diagram $\Omega_0$  of $\Omega$ with the following properties. 
 Along the sides of these $p$- and  $q$-corridors we read  words  
$r^{\beta}$  and  $\left(\lambda r\right)^{\gamma}$, respectively. (Since $\Omega$ is of minimal area these words  will be reduced and so are powers or $r$ and $\lambda r$, respectively.)  
The word one reads around the boundary circuit of $\Omega_0$ is $r^{\beta} u_0 \left(\lambda r\right)^{-\gamma} {v_0}^{-1}$ for some subwords $u_0$ and $v_0$ of some cyclic permutations of $u$ and $v$, respectively.    Killing $p$, $q$ and $r$, we get that $\lambda^{\gamma}$ equals in $\Lambda$ a word of length at most $n$, being that it is obtained from ${v_0}^{-1}u_0$ by deleting some letters.  So $\abs{\gamma} \leq  \Dist_{\langle \lambda\rangle}^{\Lambda}\mleft(n\mright)$, and therefore $\CL_{\Sigma}\mleft(u,v\mright) \leq {2} \Dist_{\langle \lambda\rangle}^{\Lambda}\mleft(n\mright) +n$, which is within the required bound \eqref{bound we want}.

\emph{Case 2: $\Omega$ has a radial $p$-corridor $C$ but no   $q$-corridor.}
Because $\Omega$ contains no  $q$-corridor, it can have no face labelled by the relator $\left[q,\lambda r\right]$.  Therefore any $r$-annulus in  $\Omega$ would be comprised of faces labelled by $\left[p,r\right]$ and $\left[\lambda, r \right]$.  But any such annulus would have  the same word along its inner and outer boundaries, so it could be excised and the boundaries glued, reducing the area of $\Omega$. (See Section~\ref{sec:removing t-rings} for more precise versions of this argument.)
Any $r$-corridor in $\Omega$ starts and ends on the boundary, so there are at most $n/2$ such corridors.
  We   claim that each $r$-corridor intersects $C$ at most twice. 
As   $C$  consists entirely of cells labelled $\left[p,r\right]$, this will then imply that $\CL_\Sigma\mleft(u,v\mright)$ is at most the length  of $C$ (i.e.\ number of cells it contains or, equivalently, the length of the word along both of its sides) plus $n$, and so at most $2n$, which will also be within the  bound \eqref{bound we want}.

To prove that an $r$-corridor can intersect $C$ at most twice, first note that the word along the side of $C$ is a power of $r$ (since $\Omega$ is minimal area and so reduced).  Therefore if an $r$-corridor   crosses $C$, it must spiral around the diagram before crossing it again.

Next suppose an $r$-corridor $D$  intersects $C$ at least three times as illustrated in the right annulus in Figure~\ref{fig:radial-spiral-intersection}. 
Then there is a sub-van~Kampen diagram (shown shaded darker in the figure) whose boundary is labelled by a word
$x_1w_1 x_2^{-1}w_2^{-1}$ where  $w_1$ and $w_2$ are words on $\set{\lambda, p}^{\pm 1}$ read along portions of the sides of $D$ and $x_1$ and $x_2$ are words on  $r^{\pm 1}$ read along  portions of the sides of $C$.  Moreover $w_1$, $w_2$, $x_1$, and $x_2$ are reduced words since $\Omega$ is minimal area (and so reduced).    
As there are no $q$-corridors, $x_1w_1x_2^{-1}w_2^{-1}=1$ in $$\Lambda'' \ := \ \langle \ A \cup \set{ \lambda, p,  r }  \, \mid     \, R \cup \set{ \, \left[\lambda, r \right], \ \left[p,r\right] \, }    \ \rangle.$$
 
On account of the retraction  $\Lambda'' \onto \langle r \rangle = \Z$ we have that $x_1 =x_2$ (as words).    
And since the letters of $w_i$ commute with the letters of $x_i$ and $\Lambda''$ also retracts onto the free product of $\Lambda$ with $\Z = \langle p \rangle$ (on killing $r$), we  deduce that $w_1=w_2$ (as words).
But then we could cut this quadrilateral out of $\Omega$ and obtain a smaller area annular diagram for $u\sim v$ by identifying the $x_1$ and $x_2$ paths and identifying the  $w_1$ and $w_2$ paths, contrary to the minimality of the area of $\Omega$.   (This is explained more precisely in Remark~\ref{cut and replace remark}.)   We conclude that indeed    an $s$-corridor can intersect $C$ at most twice.

\begin{figure}[ht]
\begin{overpic}[scale=1.0,unit=1mm]{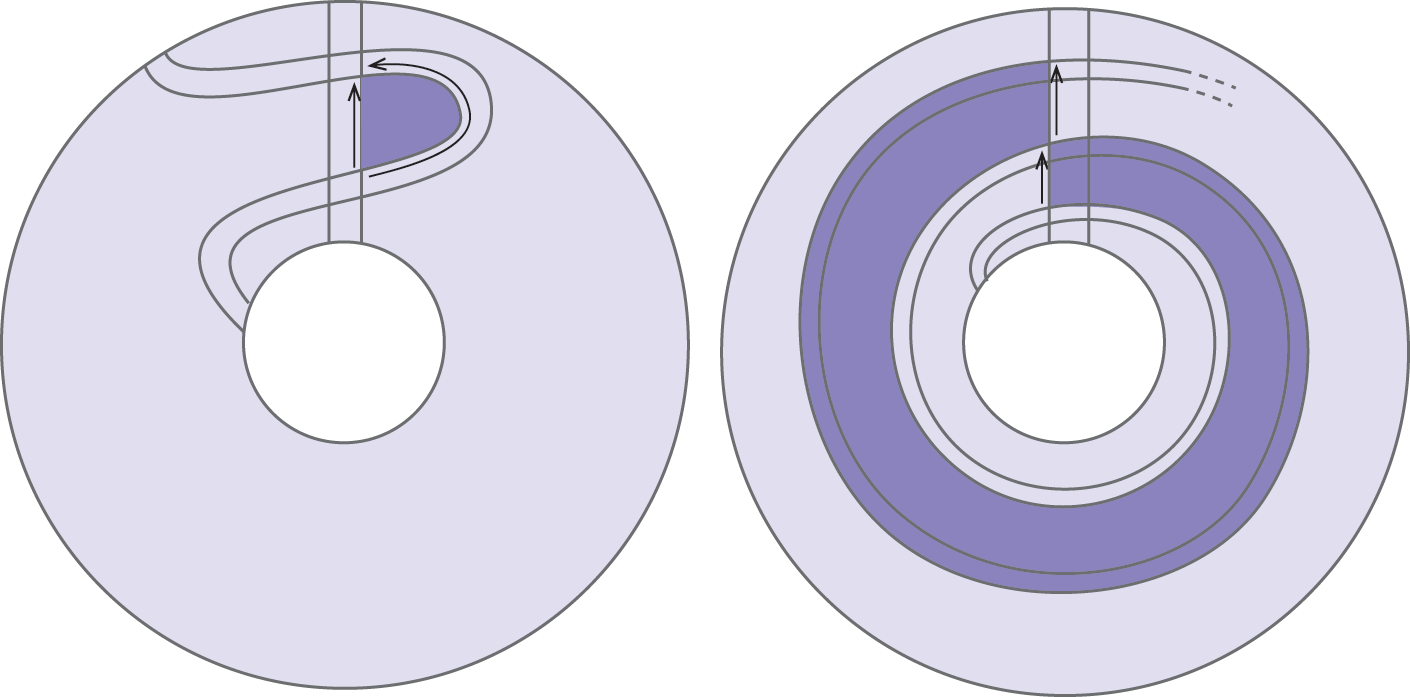}
 \put(28,36){\small{$p$}}     
 \put(28,60){\small{$p$}}     
 \put(90,36){\small{$p$}}     
 \put(90,59.5){\small{$p$}}     
 \put(28,48){\small{$x$}}     
 \put(85.5,42.7){\small{$x_1$}}     
 \put(89.6,49){\small{$x_2$}}     
 \put(40,49){\small{$y$}}     
 \put(21,31){\small{$r$}}     
 \put(11.5,55){\small{$r$}}     
 \put(83,33){\small{$r$}}     
	\end{overpic}
	\caption{Prohibited $s$-corridors in $\Omega$.}
	\label{fig:radial-spiral-intersection}
\end{figure}
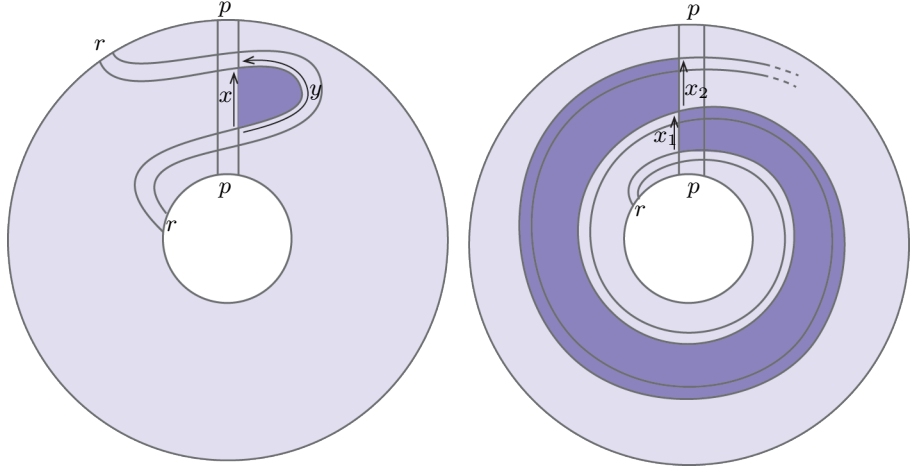

\emph{Case 3: $\Omega$  has a radial $q$-corridor but no   $p$-corridor.}   
	Let $\hat{r}  =  \lambda  r  $.   Then 	
	$$\Sigma \  = \ \langle \  A \cup { \lambda, p, q, \hat{r} }  \, \mid \, R \cup \set{ \, \left[\lambda, \hat{r} \right], \, \left[p, {\lambda}^{-1} \hat{r} \right], \  \left[q, \hat{r} \right] \, }  \ \rangle$$
	since the  relations  $\left[\lambda, r \right]=1$ and   $\left[ \lambda, \hat{r} \right]=1$ are equivalent.  Let $\hat{u}$ and $\hat{v}$ be $u$ and $v$ (respectively) with  $\lambda^{-1} \hat{r}$ in place of each letter $r$.  
	Then $\hat{u}$ and $\hat{v}$ both include the letter $q$, but not $p$.
	We can therefore find a minimal area annular diagram $\hat{\Omega}$ for $\hat{u} \sim \hat{v}$ with respect to the new presentation for $\Sigma$.
	This annular diagram $\hat{\Omega}$ will have a radial $q$-corridor, but no $p$-corridors.
	The argument now follows that given above for when $\Omega$ has a radial $p$-corridor, but no $q$-corridor, interchanging the roles of $p$ and $q$, and replacing $r$ with $\hat{r}$.  Translating back to the original presentation by replacing letters $\hat{r}$ by $\lambda r$, our length estimates may double, leading to the bound $\CL_\Sigma\left(u,v\right) \leq 4n$, which is again within   \eqref{bound we want}. 
	\end{proof}

Theorem~\ref{central distortion versus CL}  might be expected to yield finitely presented groups displaying a wide range of examples of conjugator length functions.  After all,  the set of  distortion functions of $\Z$-subgroups of finitely presented groups is broad, as we explained at the start of this section,  and there are constructions which in many cases promote examples of subgroup distortion to examples of distortion of \emph{central} $\Z$-subgroups --- see, for example, 
the passage from $G$ to $\Lambda$  in (\ref{s: exp examples}).

However, while we expect the hypothesis  $\CL_{\Lambda}   \preceq \Dist_{\langle \lambda \rangle}^{\Lambda}$ to hold in great generality, it seems tricky to verify  in specific examples.  
  To get some examples of fast-growing conjugator length functions  which we can fully identify, we turn to other constructions. 

\subsection{Exponential growing conjugator length functions} \label{s: exp examples}

In \cite{BrRi3} the first two authors build the first examples of finitely presented group whose conjugator length functions are known to grow like $n \mapsto 2^n$.   One such group is 
 $$\Lambda = \left\< a_1, a_2, a_3, s, t, \lambda  \ \left| \ \parbox{7.2cm}{$s^{-1} a_1 s= a_1a_2 a_3,  \ s^{-1} a_2 s =  a_1,  \ s^{-1} a_3 s =  a_2,$  \\    $[t,\l]=1, \  [t,a_i]=a_i\l, \ \l$  central  } \right. \right\>,$$   which is a central extension of a group $G=F(a_1, a_2, a_3) \rtimes F(s,t)$ which is in turn an HNN-extension of the hyperbolic free-by-cyclic group 
$\langle a_1, a_2, a_3, s     \mid  s^{-1} a_1 s= a_1a_2a_3,  \ s^{-1} a_2 s =  a_{1}, \ s^{-1} a_3 s =  a_2 \rangle$.  Its central subgroup $\Z \cong \< \lambda \>$  is exponentially distorted in $\Lambda$.  A classification of element-centralizers in $G$ facilitates a case-by-case analysis of conjugators whereby one can reduce the problem of bounding their lengths to problems of bounding the sizes of solutions of linear Diophantine equations.

\subsection{Fast growing conjugator length functions from fibre products} \label{sec: fast fibers}

The recent paper \cite{BrRi3} also contains examples of finitely presented groups whose conjugator length functions are comparable to fast-growing functions that are representatives of every level of the Grzegorczyk hierarchy of primitive recursive functions. The levels of the hierarchy are represented by  
Ackermann functions $A_k$, defined recursively by  $A_0(n)= n+2,\ A_1(0)=0,\ A_k(0)=2$ for $k\ge 2$, and $A_{k+1}(n + 1) = A_k(A_{k+1}(n))$ for all $k,n\ge 0$.  So $A_1(n) =2n$, $A_2(n) =2^n$, $A_3(n)$ is a height-$n$ tower of powers of $2$, and so on.   
   
For $Q$  a finitely presented group, $p: G \to Q$  an epimorphism,  
and  $P$  
the associated fibre product, the following elementary lemma connects the word problem in $Q$, the membership problem for $P<G\times G$, and the conjugacy problem in $P$.  
\begin{lemma} Let $p:G\to Q$ be an epimorphism from a  group $G$ generated by $A=\{ a_1, \ldots, a_n \}$ 
 to a finitely presented group $Q$. Let $r_1,\dots,r_m$ be words in the free group 
 $F(A)$ that generate the kernel of the composition $F(A)\to G\to Q$ as a normal subgroup; so $Q\cong \<a_1,\dots,a_n
 \mid r_1,\dots, r_m\>$. Then, the
{\em fibre product}
$$
P:= \{(x,y)\in G \times G\mid p(x)=p(y)\}
$$
is generated by $\set{(a_1, a_1), \ldots, (a_n,a_n), (r_1,1), \ldots, (r_m,1)}$.

For words $w\in F(A)$,  the condition $w = 1$ in $Q$ is equivalent to $(w,1) \in P$.
And for any element $(g,g)\in G \times G$  whose centraliser is contained in $P$,  the condition $(w,1) \in P$ is
equivalent to   $(g,g) \sim (wgw^{-1},g)$ in $P$.
\end{lemma}

Bridson \cite{Bridson13} proved a quantified  version of  this lemma in the case where $G$ is torsion-free and hyperbolic. This
relates the conjugator length function of $P$,  the geometry of cyclic subgroups in $Q$, the Dehn function of $Q$, the {\em rel-cyclics Dehn function} of $Q$ introduced in \cite{Bridson14}, and the distortion of $P$ in $G \times G$. 

To obtain the representatives of the Grzegorczyk hierarchy  that we seek,  one can take
 $Q$ to be the $k$-th  \emph{hydra group} of \cite{DR},  which has Dehn function growing $\simeq A_k$, 
 and construct an epimorphism from a torsion-free hyperbolic group $G\to Q$ with finitely generated  kernel.  In
 this case,  $P$ is  finitely presentable and the estimates from \cite{Bridson13} 
 amount to $A_k(n) \preceq \CL_{P}(n) \preceq A_k(n^2)$.  This places $\CL_{P}(n)$ between two fast-growing functions which, when $k \geq 3$, are representatives of the same level of the  hierarchy. See  \cite{BrRi3}  for details.

 \subsection{Other geometric techniques}

We have seen in this survey that the geometric tools of diagrams, corridors,  and subgroup distortion, which are  well-established in the study of Dehn functions, are also powerful in the study of conjugator length.  A further
geometric tool for studying conjugator length can be found in \cite{BrRi4}, where bounds on conjugator length are established via an analysis of the distance between the inner and outer boundary of annular diagrams.  In that
setting, the analysis reduces to an argument about how rings of quadrilaterals can fit together, 
and estimates on the \emph{skewness} of these quadrilaterals provide bounds on commutator length.   
There are various similar geometric tools that one anticipates will be useful in other settings,  but most of these
arguments are coarse-geometric in nature, and  as we noted in Section~\ref{s: not coarse}, conjugator length is not a quasi-isometry  invariant,  so we should look beyond these sorts of techniques  in search of sharper results.  The methods of the next section are part of this broader search.

\subsection{Linear Diophantine equations and $\zeta$-maps}  \label{zeta maps}

In a number of   examples, the kernel of the problem of estimating conjugator length is estimating the minimal size of a solution to a consistent system of  linear Diophantine equations.    When the system has just one equation (B\'ezout --- see Lemma~\ref{linear Diophantine equation} here) such estimates are elementary.  Theorems~\ref{Heisenberg group thm} and \ref{Stallings CL} here are examples.  When there are more equations, as for example in \cite{BrRi1}, we look  to \cite{BFRT},  where the estimates stem from  Cramer's rule.   

Linear Diophantine equations emerge in \cite{BrRi1} from comparing normal forms.  In the central-extension examples of Sections~\ref{s:snowflakes} and \ref{s: exp examples} here, they arise in the study of centralisers and what we call \emph{$\zeta$-maps}.  Centralisers are pivotal for conjugator length because if $uw_0 =w_0v$ in a group $\Gamma$, then the set of \emph{all} $w$ such that $uw =wv$ in  $\Gamma$ is the coset  $C_{\Gamma}(u)w_0$ of the centraliser of $u$ in $\Gamma$.     

Suppose   that $\Gamma$ is a central extension $$
1 \to Z \to \Gamma \to G\to 1 
$$
with $Z\cong\Z$ generated by some central element $z \in \Gamma$. Suppose $\gamma \in \Gamma$  maps to $g\in G$.
For each element of the centralizer $x\in C_{G}(g)$ 
and all preimages $\tilde{x}, \tilde{g} \in \Gamma$,
we have $\tilde{x}^{-1}\tilde{g}\tilde{x} = \tilde{g} z^{m}$ in $\Gamma$ for some integer $m$ that only depends on $g$ and $x$  because alternate choices of $\tilde{x}$ and $\tilde{g}$ differ by a power of $z$. Thereby $\zeta_g(x) :=m$ defines a  homomorphism
$$
\zeta_g : C_G(g) \to \Z
$$
 with image $\set{ N \in \Z \mid  \gamma \sim \gamma z^N}$.

Now suppose $u$ and $v$ are words such that $u \sim v$ in $\Gamma$.  Then their images $\overline{u}$ and $\overline{v}$ in $G$ are  conjugate in $G$. Let $w \in \Gamma$ be such that its image  $\overline{w}$ in $G$ is a shortest conjugator in $G$ --- that, is  $\overline{u} \overline{w}= \overline{w} \overline{v}$ and $|\overline{w}| = \CL_G(\overline{u}, \overline{v})$.  Then  $uw = wv z^N$ in $\Gamma$ for some integer $N$, and so $u \sim v z^N$ in $\Gamma$.    We can estimate this $N$ using the Dehn function of $G$ thus: $|N| \leq \Area_G(\overline{u} \overline{w}\overline{v}^{-1} \overline{w}^{-1})$ by a well known argument.  And because,   $u \sim v$ we then have $v \sim v z^N$ and so $N$ is in the image of $\zeta_v$.  And then, given a set $X= \set{x_1, \ldots, x_m}$ of generators for $C_G(v)$ with set of lifts $\overline{X} = \set{\overline{x}_1, \ldots, \overline{x}_m}$ in $\Gamma$,  any $p = x_1^{\lambda_1} \cdots x_m^{\lambda_m}$ such that 
\begin{equation} \label{e:lde} 
\zeta_v(p) = \lambda_1 \zeta_v(x_1) + \cdots + \lambda_m \zeta_v(x_m) = N
\end{equation}
 lifts to a conjugator $\overline{p}$ for $v \sim v z^N$ in $\Gamma$ and we get that 
$$\CL_{\Gamma}(u,v) \leq \dehn(n + 2 \CL_{G}(\overline{u},\overline{v})) +  |\overline{p}|_{\Gamma},$$
and we can optimize the term $|\overline{p}|_{\Gamma}$ by making a propitious choices of $\lambda_1, \ldots, \lambda_m$ according to the restrictions afforded by the linear Diophantine equation \eqref{e:lde} (cf.\ Lemma~\ref{linear Diophantine equation}) and the distortion of the subgroups  $\overline{x}_1$, \ldots, $\overline{x}_m$ in $\Gamma$.

\subsection{$S$-machines} \label{sec:S-machines}

The most comprehensive answer to the question of what functions arise as the  conjugator length functions of  finitely presented groups was recently achieved by Gillis \& Wagner \cite{GiWa}.  They took up the challenge, put forward by the second author at the GAGTA 2025 conference, of constructing wide-ranging examples by using the $S$-machine technology that originates in \cite{SBR}.

Gillis \& Wagner showed that for all real numbers $\alpha \geq  2$ whose decimal expansion is  computable in double-exponential time, there exists a finitely presented group whose conjugator length function grows like $n^{\alpha}$.  More generally, any function that grows at least quadratically and can be realized as the time function of an $S$-machine is the conjugator length function (up to $\simeq$) of a finitely presented group $G$. Indeed, the set of conjugator length functions  of  finitely presented groups includes (up to $\simeq$) the set of all Dehn functions.  In a way made precise in \cite{SBR}, that set is close to being the set of all time-functions of Turing machines.
Here's an outline of Gillis \& Wagner's  approach.    

$S$-machines are groups that have coded into their finitely presentations, certain non-deterministic tape-computation machines which are comparable to multi-tape Turing machines.  These machines are realized   as  multiple HNN-extension of  free groups, so that  runs of tape-computations are displayed within van~Kampen diagrams as stacks of corridors.  In \cite{SBR}  and many  subsequent $S$-machine designs, these stacks appear as `sectors' within  systems of concentric annular corridors   arranged around a central `hub' relation which corresponds to the machine's accept state.    This leads to time and space complexity of the computing machine being related to the Dehn function of $G$.     

Gillis \& Wagner's  $S$-machines of \cite{GiWa} follow a line of $S$-machines that began with \cite{OS5}; they absent the hubs rendering these diagrams annular.   The run times of the machines are the heights of the stacks, and so give the conjugator length as witnessed in these annular diagrams.     

In common with other papers on $S$-machines,  delivering on the outline strategy requires a formidable technical struggle.  In order to bound  conjugator length for the groups, some  understanding of diagrams for \emph{all} conjugate pairs of words is required and there is wide scope for pathologies. For example, instead of forming concentric annuli, corridors might spiral, a  possibility which   Gillis \& Wagner curb by arranging for secondary computations to take place so as to give rise to what they call `history sectors' in the annular diagrams.        

Gillis \& Wagner have recently taken their techniques in another direction in \cite{GiWa2} for further breakthroughs which address questions we raised in earlier versions of this survey --- we referenced these in the contexts of earlier discussions in Sections~\ref{the CP and algorithms} and \ref{sec: rel to other invariants}.  They construct finitely presented groups $G$ with large Dehn functions but tightly constrained conjugator length functions.  The most extreme instance is an example where the word problem is undecidable (and so the same is true of the conjugacy problem), and yet the conjugator length function is quadratic.   

In this $S$-machine construction the hub relations are present.  The underlying computational machine runs the halting problem (say) and, the undecidability of that problem conveys to the word problem in $G$.  However, the construction is such that long computations (i.e., a high stack of corridors) only occur starting from a word $u$ if that computation halts, which is to say that $u$ is conjugate to a hub relation and therefore $u=1$  in $G$.  So if a word  $v$ satisfies $u \sim v$ in  $G$,   then $u=v=1$  in $G$, and so $\CL(u,v) =0$.  If, on the  other hand, $u \neq 1$, then the computation is short in such a way that there must be an annular diagram for $u$ and $v$ with small, indeed quadratically bounded, conjugator length.

\bibliographystyle{alpha}
\bibliography{bibli}

\ni {Martin R.\ Bridson} \\
Mathematical Institute, Andrew Wiles Building, Oxford OX2 6GG, United Kingdom. {bridson@maths.ox.ac.uk}, \
\href{http://www2.maths.ox.ac.uk/~bridson/}{https://people.maths.ox.ac.uk/bridson/}

\ni  {Timothy R.\ Riley} \rule{0mm}{6mm} \\
Department of Mathematics, 310 Malott Hall,  Cornell University, Ithaca, NY 14853, USA.  {tim.riley@math.cornell.edu}, \
\href{http://www.math.cornell.edu/~riley/}{http://www.math.cornell.edu/$\sim$riley/}
 
  \ni {Andrew W.\ Sale} \rule{0mm}{6mm} \\
  Department of Mathematics,
  University of Hawaii at Manoa,
  Honolulu, HI 96822, USA.   
  \href{http://math.hawaii.edu/~andrew/}{http://math.hawaii.edu/$\sim$andrew/}

 \end{document}